\documentclass[a4paper,reqno,11pt]{amsart}
\usepackage{amsfonts}
\usepackage{amsmath}
\usepackage{amssymb}
\usepackage{graphicx}
\usepackage{hyperref}
\usepackage{subfigure}
\usepackage{enumerate}

\newtheorem{theorem}{Theorem}[section]
\newtheorem{proposition}[theorem]{Proposition}
\newtheorem{lemma}[theorem]{Lemma}
\newtheorem{corollary}[theorem]{Corollary}

\newenvironment{claim}{\smallskip\noindent \textbf{Claim.}\rm}{\smallskip}
\newenvironment{remark}{\smallskip\noindent \textbf{Remark.}\rm}{\smallskip}

\newenvironment{acknowledgment}{\smallskip{\sc Acknowledgments.}\rm}{\smallskip}
\renewenvironment{proof}[1][Proof]{\textbf{#1.} }{\ \rule{0.5em}{0.5em}}
\renewcommand{\theequation}{\thesection.\arabic{equation}}
\newenvironment{notation}{\smallskip{\sc Notation.}\rm}{\smallskip}
\numberwithin{equation}{section}
\input{tcilatex}

\def\RM{\rm}

\def\func#1{\mathop{\mathrm{#1}}\nolimits}

\def\enddoc{\end{document}}

\def\FRAME#1#2#3#4#5#6#7#8
{
 \begin{figure}[ptbh]
 \begin{center}
 \includegraphics[height=#3]{#7}
 \caption{#5}
 \label{#6}
 \end{center}
 \end{figure}
}

\def\Eins{\mathbf{1}}

\begin{document}
\title[Heat kernel estimates and isoperimetric inequalities]{Heat kernel
estimates for an operator with a singular drift and isoperimetric
inequalities}
\author{Alexander Grigor'yan}
\address{Department of Mathematics, University of Bielefeld, 33501
Bielefeld, Germany}
\email{grigor@math.uni-bielefeld.de}
\author{Shunxiang Ouyang}
\address{Department of Mathematics, University of Bielefeld, 33501
Bielefeld, Germany}
\email{souyang@math.uni-bielefeld.de}
\author{Michael R\"ockner}
\address{Department of Mathematics, University of Bielefeld, 33501
Bielefeld, Germany}
\email{roeckner@math.uni-bielefeld.de}
\thanks{Supported by SFB 701 of the German Research Council}
\date{\today }
\subjclass[2010]{58J35, 26D10, 28A75}
\keywords{Isoperimetric inequalities, functional inequalities, weighted
measure, singular drift}

\begin{abstract}
We prove upper and lower bounds of the heat kernel for the operator $\Delta
-\nabla (\frac{1}{|x|^{\alpha }})\cdot \nabla $ in $\mathbb{R}^{n}\setminus
\left\{ 0\right\} $ where $\alpha >0$. We obtain these bounds from an
isoperimetric inequality for a measure $\mathrm{e}^{-\frac{1}{|x|^{\alpha }}%
}dx$ on $\mathbb{R}^{n}\setminus \{0\}$. The latter amounts to a certain
functional isoperimetric inequality for the radial part of this measure.
\end{abstract}

\maketitle
\tableofcontents

\section{Introduction}

Consider the following differential operator $\mathcal{L}=\Delta +\nabla
\psi \cdot \nabla $ defined on $M:=\mathbb{R}^{n}\setminus \{0\}$, with a
singular potential 
\begin{equation*}
\psi (x)=-\frac{1}{|x|^{\alpha }},\quad \alpha >0.
\end{equation*}

The purpose of this paper is to obtain uniform bounds for the heat kernel $%
p_{t}(x,y)$ of $\mathcal{L}$ that would take into account the singularity of 
$\psi $ at the origin. In order to define what is the heat kernel of $%
\mathcal{L}$ let us observe that $\mathcal{L}$ can be written in the form 
\begin{equation*}
\mathcal{L}=\mathrm{e}^{-\psi }\mathrm{div}(\mathrm{e}^{\psi }\nabla )
\end{equation*}%
which implies that $\mathcal{L}$ is symmetric with respect to the following
measure: 
\begin{equation}
d\mu (x)=\mathrm{e}^{\psi (x)}\,dx=\mathrm{e}^{-\frac{1}{|x|^{\alpha }}}dx.
\label{mu}
\end{equation}%
That is, the operator $\mathcal{L}$ is formally self-adjoint on $%
L^{2}=L^{2}(M,\mu )$. Following the terminology of \cite{Gri06_JW06}, $%
\mathcal{L}$ is the Laplace operator of the weighted manifold\footnote{%
A weighted manifold is a couple $\left( M,\mu \right) $ where $M$ is a
Riemannian manifold and $\mu $ is a measure on $M$ with a smooth positive
density with respect to the Riemannian measure.} $(M,\mu )$. Using the
Friedrichs extension of this operator, one defines the associated heat
semigroup $P_{t}=\mathrm{e}^{t\mathcal{L}}$, $t\geq 0$, acting in $L^{2}$.
The heat kernel of $\mathcal{L}$ is then the integral kernel of $P_{t}$,
that is, a function $p_{t}(x,y)$ defined on $\mathbb{R}_{+}\times M\times M$
such that, for all $f\in L^{2}$, $t\geq 0$, $x\in M$, 
\begin{equation*}
P_{t}f(x)=\int_{M}p_{t}(x,y)f(y)\,d\mu (y).
\end{equation*}%
By general regularity theory, the heat kernel always exists and is a smooth
positive function of $(t,x,y)$ (cf. \cite{Gri09_book}).

The motivation for considering heat kernels of operators as $\mathcal{L}$
with singular drift comes from \cite{KR05}, where global existence and
uniqueness of strong solutions for stochastic differential equations (SDE)
with singular drifts was proved. The most important applications are the
analysis of particle systems with physically realistic, hence singular
interactions (cf. \cite[Section 9]{KR05}). One example is a diffusion in a
frozen random environment given by a countable set $\gamma$ of particles in $%
\mathbb{R}^n$, distributed according to a Ruelle Gibbs measure, i.e. the
diffusion solves the SDE 
\begin{equation*}
dX(t)=b(X(t))dt+dW(t), 
\end{equation*}
with 
\begin{equation*}
b(x):=-\sum_{y\in \gamma} \nabla V(x-y),\quad x\in\mathbb{R}^n, 
\end{equation*}
and $V\colon \mathbb{R}^n\to\mathbb{R}$ is a pair potential describing the
interaction of the moving particle $X(t)$, $t\geq 0$, with those in $\gamma$%
. $V$ is typically very singular at $x=0$ (e.g. of Lenard-Jones type)
modelling the strong repulsion between two particles. One of the main and
most interesting open questions about the solution $X(t)$, $t\geq 0$, is
whether (depending on the location of the points in $\gamma$ and the
strength of the singularity of $V$) it exhibits sub- or super-diffusive
behavior. So, a good way to start is to examine the heat kernel of the
corresponding generator $\mathcal{L}_b=\Delta +\langle b,\nabla \rangle$,
which is symmetric on $L^2\left(\mathbb{R}^n,\exp\left(-\sum_{y\in \gamma}
V(x-y) dx\right)\right)$. Therefore, in this paper, as a first step, we
study the model case described above, where $b=\nabla\psi$ and we have only
one particle, i.e. $\gamma=\{0\}$.

Our main results --- Theorems \ref{Thm:upper-bound-hk} and \ref%
{Thm:sup-lower-heat-kernel-apply} below, provide the following bounds for
the heat kernel of $\mathcal{L}$ for all $0<t<1:$ 
\begin{equation}
\sup_{x,y}p_{t}(x,y)\leq C\exp \left( \frac{C}{t^{\frac{\alpha }{\alpha +2}}}%
\right)  \label{ptup}
\end{equation}%
and 
\begin{equation*}
\sup_{x}p_{t}(x,x)\geq c\exp \left( \frac{c}{t^{\frac{\alpha }{\alpha +2}}}%
\right)
\end{equation*}%
where $C,c$ are some positive constants. It is important that these
estimates correctly capture the term $\exp \left( \frac{\func{const}}{t^{%
\frac{\alpha }{\alpha +2}}}\right) $, describing the short time on-diagonal
behavior of the heat kernel, that is determined by the singularity of the
drift.

Presently a variety of methods are available for obtaining heat kernel
estimates. A challenging feature of the above problem is that the methods
based on the curvature bounds fail here (cf. \cite{Ouyang09}). We use
instead the approach developed by the first-named author \cite%
{Gri09_book,Gri94_RMI,Gri06_JW06} that is based on isoperimetric and
Faber-Krahn inequalities. Given a weighted manifold $\left( M,\mu \right) $
and a function $\Lambda :(0,+\infty )\rightarrow \lbrack 0,+\infty ),$ we
say that $(M,\mu )$ satisfies the Faber-Krahn inequality with function $%
\Lambda $ if, for any precompact open set $U\subset M$, the following
inequality holds 
\begin{equation}
\lambda _{1}(U)\geq \Lambda (\mu (U)),  \label{Lambda-FK:ineq-iso}
\end{equation}%
where $\lambda _{1}(U)$ denotes the bottom of the spectrum of $\mathcal{L}$
in $L^{2}(U,\mu )$ with the Dirichlet boundary condition on $\partial U$. By
a result of \cite{Gri94_RMI}, the Faber-Krahn inequality implies a certain
upper bound of the heat kernel. On the other hand, by Cheeger's inequality,
the Faber-Krahn inequality (\ref{Lambda-FK:ineq-iso}) follows from a certain
isoperimetric inequality of the form 
\begin{equation}
\mu ^{+}(U)\geq J(\mu (U)),  \label{muJ}
\end{equation}%
where $\mu ^{+}(U)$ is the perimeter of $U$ defined by 
\begin{equation*}
\mu ^{+}(A)=\liminf_{r\rightarrow 0^{+}}\frac{\mu (A^{r})-\mu (A)}{r},
\end{equation*}%
where $A^{r}$ is the $r$-neighborhood of $A$ with respect to the Riemannian
metric of $M$. Any function $J$ that satisfies (\ref{muJ}) is called a lower
isoperimetric function of the measure $\mu $. Our main technical result,
Theorem \ref{Thm:iso-ineq-global}, yields the following lower isoperimetric
function of $\mu $: 
\begin{equation*}
J(v)=C\, v \left( \log \frac{1}{v}\right) ^{1+\frac{1}{\alpha }}
\end{equation*}%
for small enough values of $v$ and for some constant $C=C(n,\alpha )>0$.
This estimate leads in the end to the upper bound (\ref{ptup}) of the heat
kernel.

Let us recall some previous results on isoperimetric inequalities (for more
information on this active field, we refer the reader to \cite%
{Bar02,BH97,Hue11,Ros05} and the references therein). For any weighted
manifold $\left( M,\mu \right) $ let $I_{\mu }$ denote the isoperimetric
function of $\mu $, that is, the largest possible lower isoperimetric
function. For some specific measures on Euclidean space, the respective
isoperimetric functions are known exactly. For example, the isoperimetric
function for the Lebesgue measure $\lambda $ in $\mathbb{R}^{n}$ is given by 
\begin{equation*}
I_{\lambda }(v)=n\omega _{n}^{1/n}v^{(n-1)/n},
\end{equation*}%
where $\omega _{n}$ is the $\left( n-1\right) $-volume of the unit sphere $%
\mathbb{S}^{n-1}$ in $\mathbb{R}^{n}$.

Due to the celebrated result of Borell \cite{Bor75} and Sudakov-Tsirel'son 
\cite{ST74}, the isoperimetric function for the Gaussian measure 
\begin{equation*}
\gamma ^{n}(dx)=(2\pi )^{-\frac{n}{2}}\exp \left( {-\frac{|x|^{2}}{2}}%
\right) dx
\end{equation*}%
is given by 
\begin{equation*}
I_{\gamma ^{n}}(v)=c~\left( v\wedge (1-v)\right) \sqrt{\log \frac{1}{v\wedge
(1-v)}},
\end{equation*}%
where $c>0$ is some constant independent of $n$.

Various generalizations of this result have been studied. In particular, in 
\cite{Hue11} a lower bound is given for the isoperimetric function of the
probability measure 
\begin{equation}
\nu^{n,\alpha}(dx):=\frac{1}{Z_{n,\alpha}}\mathrm{e}^{-|x|^\alpha}\,dx
\label{nuna}
\end{equation}
on $\mathbb{R}^n$ with $\alpha\geq 1$ (where $Z_{n,\alpha}$ is a
normalization constant):

\begin{equation*}
I_{\nu^{n,\alpha}}(v)\geq C n^{\frac12-\frac1\alpha} (v\wedge (1-v)) \left(
\log\frac{1}{v\wedge (1-v)} \right)^{1- \frac{1}{\alpha\wedge 2} }
\end{equation*}
for some constant $C>0$ independent of $n$.

Note that all measures in $\mathbb{R}^{n}$ mentioned above are spherically
symmetric, so that they can be split into a product of a one dimensional
measure in the radial direction and the canonical measure on $\mathbb{S}%
^{n-1}$ in the angular direction. The isoperimetric function of the measure
on $\mathbb{S}^{n-1}$ is classical. The isoperimetric inequality for the
radial part of the measure $\mu $ is also straightforward. Gluing the radial
and angular isoperimetric inequalities presents certain challenges. For that
purpose, we use a so called functional isoperimetric inequality that was
proved for the Gaussian measure by Bobkov \cite{Bob96} and for the measure (%
\ref{nuna}) by Huet \cite{Hue11}. This inequality enjoys the following
distinctive feature: if it is known in the radial and angular directions, it
implies easily an isoperimetric inequality in the whole $\mathbb{R}^{n}$.

Hence, the last problem that we face on this long road to the goal is
obtaining the functional isoperimetric inequality for the radial part of the
measure $\mu $ (for the spherical part it follows from \cite{Hue11}). The
methods previously used for the measures $\gamma ^{n}$ and $\nu ^{n,\alpha }$
do not work for the measure $d\mu =\mathrm{e}^{-\frac{1}{|x|^{\alpha }}}dx$,
as they require the measure $\mu $ to be finite. We have developed an
entirely new method that constitutes the most interesting part of this paper
and is presented in Theorem \ref{Thm:iso-func-ineq} (and its application to
the measure $\mu $ is given in Theorem \ref{Thm:iso-func-ineq:radial}).

The organization of this paper follows the above scheme of the proof. In
Section \ref{Sec:1dim-general-ineq} we deduce a functional isoperimetric
inequality for measures on $\mathbb{R}_{+}$ from the normal isoperimetric
inequality. In Section \ref{Sec:ineq:radial} we obtain the functional
isoperimetric inequality for the radial part of the measure $\mu $. In
Section \ref{Sec:ineq:angular} we verify the functional isoperimetric
inequality for the canonical measure on the unit sphere. In Section \ref%
{Sec:ineq:combine} we combine these two inequalities to obtain a full
functional isoperimetric inequality for the measure $\mu $ and, hence, the
isoperimetric inequality for $\mu .$ Finally, in Section \ref{secUpper} we
apply our isoperimetric inequality to obtain the heat kernel upper estimate,
and in Section \ref{SecLow} we prove the lower estimate.

\begin{notation}
1. For any two non-negative functions $f,g$, the relation $f\approx g$ means
that $f$ and $g$ are comparable, that is, there exists a constant $C>0$ such
that 
\begin{equation*}
\frac{1}{C}g\leq f\leq Cg
\end{equation*}%
for a specified range of the arguments of $f,g$.

2. Letter $C$, $C_{1},C_{2},C^{\prime }$ etc. are used to denote various
positive constants whose values can change at each occurrence, unless
otherwise specified.

3. We frequently use the function $I\left( v\right) =v\left( \log \frac{1}{v}%
\right) ^{\beta }$ defined for $0<v\leq 1.$ Since $\lim_{v\rightarrow
0}I\left( v\right) =0$, we always assume without further explanation that
this function is extended to all $0\leq v\leq 1$ by setting $I\left(
0\right) =0$.
\end{notation}

\section{One-dimensional functional isoperimetric inequalities}

\label{Sec:1dim-general-ineq} In this section we prove the following theorem
that is the key to our main result.

\begin{theorem}
\label{Thm:iso-func-ineq} Let $\phi\colon \mathbb{R}_+\to \mathbb{R}_+$ be a
non-negative continuous function on $\mathbb{R}_+$ and consider the Borel
measure $d\nu(r)=\phi(r)\,dr$ on $\mathbb{R}_+$. 
Let $I,J,K,L$ be four non-negative functions on $\mathbb{R}_+$ with the
following properties:

\begin{enumerate}
\item[$\left( i\right) $] For all $a,b\geq 0$, 
\begin{equation}
I(ab)\leq bJ(a)+K(aL(b));  \label{Equ:IJKL-ineq}
\end{equation}

\item[$\left( ii\right) $] $J$ is a lower isoperimetric function for the
measure $\nu $;

\item[$\left( iii\right) $] $K$ is non-decreasing and concave;

\item[$\left( iv\right) $] $L$ is concave.
\end{enumerate}

Then, for all nonnegative continuously differentiable functions $f$ on $%
\mathbb{R}_+$ with bounded support, we have 
\begin{equation}  \label{Equ:iso-func-ineq-1dim-general}
I\left( \int_{\mathbb{R}_+} f\,d\nu \right) \leq K\left( \int_{\mathbb{R}_+}
L(f)\,d\nu\right) +\int_{\mathbb{R}_+}|f^\prime|\,d\nu.
\end{equation}
\end{theorem}

\begin{remark}
The conditions and statement of Theorem \ref{Thm:iso-func-ineq} are similar
to that of \cite[Theorem 2]{Bar02}. The difference is that \cite[Theorem 2]%
{Bar02} works with probability measure, while the measure $\nu$ in Theorem %
\ref{Thm:iso-func-ineq} is general. The method of the proof for \cite[%
Theorem 2]{Bar02} does not work in our setting, and our proof is based on an
entirely different approach.
\end{remark}

In this paper we shall only use the special case of Theorem \ref%
{Thm:iso-func-ineq} when $J=L=\mathrm{const}\, I$ and $K=\mathrm{id. }$ For
convenience of the reader, let us state Theorem \ref{Thm:iso-func-ineq} in
this case.

\begin{theorem}
\label{Thm:iso-func-ineq-special} Let $\phi\colon \mathbb{R}_+\to \mathbb{R}%
_+$ be a non-negative continuous function on $\mathbb{R}_+$ and consider the
Borel measure $d\nu(r)=\phi(r)\,dr$ on $\mathbb{R}_+$. Let $I$ be a
non-negative function on $\mathbb{R}_+$ with the following properties:

\begin{enumerate}
\item[$\left( i\right) $] For some constant $C>0$ and for all $a,b\geq 0$, 
\begin{equation}
CI(ab)\leq bI(a)+aI(b).  \label{Equ:I-ineq}
\end{equation}

\item[$\left( ii\right) $] $I$ is a concave lower isoperimetric function for 
$\nu $.
\end{enumerate}

Then, for all non-negative continuously differentiable functions $f$ on $%
\mathbb{R}_+$with bounded support, we have 
\begin{equation}  \label{Equ:iso-func-ineq-1dim-generalII}
C I\left( \int_{\mathbb{R}_+} f\,d\nu \right) \leq \int_{\mathbb{R}_+}
I(f)\,d\nu +\int_{\mathbb{R}_+}|f^\prime|\,d\nu.
\end{equation}
\end{theorem}

The proof of Theorem \ref{Thm:iso-func-ineq} will consist of a series of
lemmas. In fact, we shall prove an extension of (\ref%
{Equ:iso-func-ineq-1dim-general}) for a class of step functions $f$. Let $f$
be a real-valued function on $\mathbb{R}_+$ with bounded support. Define the 
\emph{weighted total variation} of $f$ with respect to the measure $\nu$ by 
\begin{equation*}
V_\nu(f)=\sup_{\{\xi_0,\xi_1,\cdots,\xi_{n}\}
}\sum_{k=1}^{n}|f(\xi_{k})-f(\xi_{k-1})| \phi(\xi_{k-1}),
\end{equation*}
where $\sup$ is taken over all finite increasing sequences $%
\{\xi_0,\xi_1,\cdots,\xi_{n}\}$ of non-negative reals with arbitrary $n\in%
\mathbb{N}$ such that $\mathrm{supp}f\subset[\xi_{0},\xi_{n}]$. For example,
if $f$ is continuously differentiable then 
\begin{equation*}
V_\nu (f)=\int_{\mathbb{R}_+}|f^\prime|\,d\nu.
\end{equation*}
A function $f$ on $\mathbb{R}_{+}$ is called an \emph{elementary step function%
} if it has the form 
\begin{equation*}
f=b \Eins_{[r,s)}
\end{equation*}
for some real constant $b$ and $0\leq r< s$. A function $f$ on $\mathbb{R}_+$
is called a \emph{step function} if it is a finite sum of elementary step
functions. Clearly, any step function can be represented in the following
form 
\begin{equation}  \label{Equ:step-fnt-def}
f=\sum_{k=1}^n b_k\Eins_{[x_{k-1},x_k)},
\end{equation}
where $0= x_0< x_1< x_2<\cdots< x_n$, and $b_k$ are real constants. For the
step function (\ref{Equ:step-fnt-def}) we obviously have 
\begin{equation*}
V_\nu(f)=\sum_{k=1}^{n} |b_{k+1}-b_{k}| \phi(x_{k}),
\end{equation*}
where we set $b_{n+1}=0$.

For the proof of Theorem \ref{Thm:iso-func-ineq}, we shall first prove that
any non-negative step function $f$ satisfies the following inequality 
\begin{equation}  \label{Equ:iso-func-ineq-1dim-general-more}
I\left( \int_{\mathbb{R}_+} f\,d\nu \right) \leq K\left(\int_{\mathbb{R}_+}
L(f)\,d\nu\right) +V_{\nu}(f).
\end{equation}
We start with elementary step functions.

\begin{lemma}
\label{Lem:func-ineq-for-1-step-f} Under the hypotheses of Theorem \emph{\ref%
{Thm:iso-func-ineq},} inequality \emph{(\ref%
{Equ:iso-func-ineq-1dim-general-more})} holds for any elementary step
function of the form $f=b\Eins_{[r,s)}$, where $b\geq0$ and $0\leq r< s$.
\end{lemma}

\begin{proof}
Let $a=\nu([r,s))$. It is clear that 
\begin{equation*}
I\left( \int_{\mathbb{R}_+} f\,d\nu \right) = I(b\nu([r,s)))=I(a b)
\end{equation*}
and 
\begin{equation*}  \label{Equ:I=;K=}
K\left(\int_{\mathbb{R}_+} L(f)\,d\nu\right) \geq K(L(b)\nu([r,s)))=K(a
L(b)).
\end{equation*}
Using that $J $ is a lower isoperimetric function for $\nu$, we obtain, for
the case $r>0$ 
\begin{equation*}
V_{\nu}(f) =b\left(\phi(r)+\phi(s)\right) = b\nu^+([r,s)) \geq b
J\left(\nu([r,s))\right) = b J(a),
\end{equation*}
and for the case $r=0$ 
\begin{equation*}
V_{\nu}(f) =b\phi(s) = b\nu^+([0,s)) \geq b J\left(\nu([0,s))\right) = b
J(a).
\end{equation*}
Hence, (\ref{Equ:iso-func-ineq-1dim-general-more}) follows from (\ref%
{Equ:IJKL-ineq}).
\end{proof}

\medskip

Before we can treat an arbitrary step function, let us prove the following
lemma.

\begin{lemma}
\label{Lem:f_k->step_f} Let $f_1,f_2,\cdots,f_n$ be non-negative functions
on $\mathbb{R}_{+}$ with bounded supports such that \emph{(\ref%
{Equ:iso-func-ineq-1dim-general-more})} holds for all $f_k$, $%
k=1,2,\cdots,n. $ Assume also that 
\begin{equation*}
\int_{\mathbb{R}_+} f_1\,d\nu=\int_{\mathbb{R}_+} f_2\,d\nu=\cdots=\int_{%
\mathbb{R}_+} f_n\,d\nu.
\end{equation*}
Choose a sequence $\{p_{k}\}^{n}_{k=1}$ of non-negative reals such that $%
\sum_{k=1}^np_k=1,$ and set 
\begin{equation}  \label{pk}
f=\sum_{k=1}^n p_k f_k.
\end{equation}
If 
\begin{equation}  \label{Equ:equal-variation-assump}
V_\nu(f)=\sum_{k=1}^n p_k V_\nu(f_k),
\end{equation}
then \emph{(\ref{Equ:iso-func-ineq-1dim-general-more})} holds also for $f$.
\end{lemma}

Note that (\ref{pk}) implies the inequality 
\begin{equation*}
V_\nu(f)\le\sum_{k=1}^n p_k V_\nu(f_k),
\end{equation*}
whereas the equality (\ref{Equ:equal-variation-assump}) holds only in
specific situations, one of which will be described below.

\begin{proof}
It is clear that, for all $k=1,2,\cdots,n$, we have 
\begin{equation*}
\int _{\mathbb{R}_{+}}f_k\,d\nu=\int _{\mathbb{R}_{+}}f\,d\nu.
\end{equation*}
By hypotheses, we have, for all $k=1,2,\cdots,n$, 
\begin{equation*}
I\left(\int_{\mathbb{R}_+} f_k\,d\nu \right) \leq K\left(\int_{\mathbb{R}_+}
L(f_k)\,d\nu \right) + V_\nu(f_k).
\end{equation*}
Using the monotonicity of the function $K$, the concavity of $K$ and $L$,
and (\ref{Equ:equal-variation-assump}) we obtain 
\begin{equation*}
\begin{split}
I\left(\int_{\mathbb{R}_+}f\,d\nu\right) &=\sum_{k=1}^n p_k I\left(\int_{%
\mathbb{R}_+} f_k \, d\nu\right) \\
&\leq \sum_{k=1}^n p_k \left( K \left(\int_{\mathbb{R}_+} L(f_k) \,
d\nu\right) + V_\nu(f_k) \right) \\
&\leq K\left( \int_{\mathbb{R}_+} L\left(\sum_{k=1}^n p_k f_k\right)\,d\nu
\right) +\sum_{k=1}^n p_k V_\nu(f_k) \\
&= K\left( \int_{\mathbb{R}_+} L (f)\,d\nu \right) + V_\nu(f),
\end{split}%
\end{equation*}
which was to be proved.
\end{proof}

\begin{lemma}
\label{Lem:step-f-rewrite} Let $f$ be a step function of the following form 
\begin{equation}  \label{Equ1:Lem:step-f-rewrite}
f=\sum_{k=1}^n b_k\Eins_{[x_{k-1},x_k)},
\end{equation}
where $0=x_0<x_1<x_2<\cdots<x_n$ and $b_k\geq 0$ for all $k=1,2,\cdots,n$.
Then $f$ can be represented in the form 
\begin{equation}  \label{Equ:f=sum_fk}
f=\sum_{k=1}^n p_k f_k,
\end{equation}
where each $f_{k}$ is a non-negative elementary function and the following
relations are satisfied: 
\begin{equation}  \label{Equ:sum_pk=1}
\sum_{k=1}^n p_k =1,\quad p_k\geq 0,
\end{equation}
\begin{equation}  \label{Equ:int_fk=int_f}
\int_{\mathbb{R}_+} f_k \,d\nu=\int _{\mathbb{R}_{+}}f \,d\nu,
\end{equation}
and 
\begin{equation}  \label{Equ:V_f=sum_V_fk}
V_\nu(f)=\sum_{k=1}^n p_k V_{\nu}(f_k).
\end{equation}
\end{lemma}

\begin{proof}
In the case $n=1$ we just need to take $f_1=f$ and $p_{1}=1$. Assume that $%
n>1 $ and make the induction step from $n-1$ to $n$. We can assume that $f$
as in (\ref{Equ1:Lem:step-f-rewrite}) is not elementary. For convenience, we
set $b_0=b_{n+1}=0$. Let $b_{k_0}$ be the maximal value of $\{b_k\colon
k=1,2,\ldots,n\}.$ Without loss of generality we assume that 
\begin{equation*}
b_{k_{0}-1}\leq\, b_{k_{0}+1}
\end{equation*}
because the case when $b_{k_{0}-1}\geq b_{k_{0}+1}$ can be treated
similarly. If $b_{k_{0}+1}=b_{k_{0}}$, then we can reduce the number of
intervals and use the inductive hypothesis. Hence, we can assume that 
\begin{equation*}
b_{k_{0}+1}<b_{k_{0}}.
\end{equation*}
Let us define a function $h$ as follows 
\begin{equation}  \label{Equ:def:funct-h-1}
h=f\Eins_{\mathbb{R}_{+}\setminus [x_{k_{0}-1},x_{k_{0}})}+b_{k_{0}+1}\Eins%
_{[x_{k_{0}-1},x_{k_{0}})},
\end{equation}
that is, $h$ is equal to $f$ outside $[x_{k_{0}-1},x_{k_{0}})$ and is equal
to $b_{k_0+1}$ on $[x_{k_{0}-1},x_{k_{0}})$.

Define also a function $g$ by 
\begin{equation*}
g=c\Eins_{[x_{k_{0}-1},x_{k_{0}})},
\end{equation*}
where the constant $c$ is chosen to satisfy the following condition 
\begin{equation}  \label{Equ:g-f-to-c}
\int_{\mathbb{R}_{+}}g\,d\nu =\int_{\mathbb{R}_{+}}f\,d\nu ,
\end{equation}
that is, 
\begin{equation*}
c=\frac{1}{\nu ([x_{k_{0}-1},x_{k_{0}}))}\int_{\mathbb{R}_{+}}f\,d\nu
=b_{k_0}+\frac{1}{\nu ([x_{k_{0}-1},x_{k_{0}}))} \int_{\mathbb{R}%
_{+}\setminus (x_{k_{0}-1},x_{k_{0}}]}f\,d\nu .
\end{equation*}%
It is clear that $c> b_{k_0}$ since outside $[x_{k_{0}-1},x_{k_{0}})$ the
function $f\geq 0$ is not identically zero.

It follows that 
\begin{equation*}
g>f\ \text{on}\ [x_{k_{0}-1},x_{k_{0}}).
\end{equation*}%
On the other hand, we have 
\begin{equation*}
f=b_{k_{0}}>b_{k_{0}+1}=h\ \text{on}\ [x_{k_{0}-1},x_{k_{0}}).
\end{equation*}%
Hence, we obtain 
\begin{equation*}
g>f>h\ \text{on}\ [x_{k_{0}-1},x_{k_{0}}).
\end{equation*}%
Therefore, there is a constant $p\in (0,1)$ such that 
\begin{equation}
f=pg+h  \label{Equ:f=pg+h}
\end{equation}%
on $[x_{k_{0}-1},x_{k_{0}})$. Noting that $h=f$ and $g=0$ outside $%
[x_{k_{0}-1},x_{k_{0}})$, we see that (\ref{Equ:f=pg+h}) holds on $\mathbb{R}%
_{+}$. 
\begin{figure}[tbph]
\begin{center}
\subfigure[Step function $f$]{
    \includegraphics[width=0.43\textwidth]{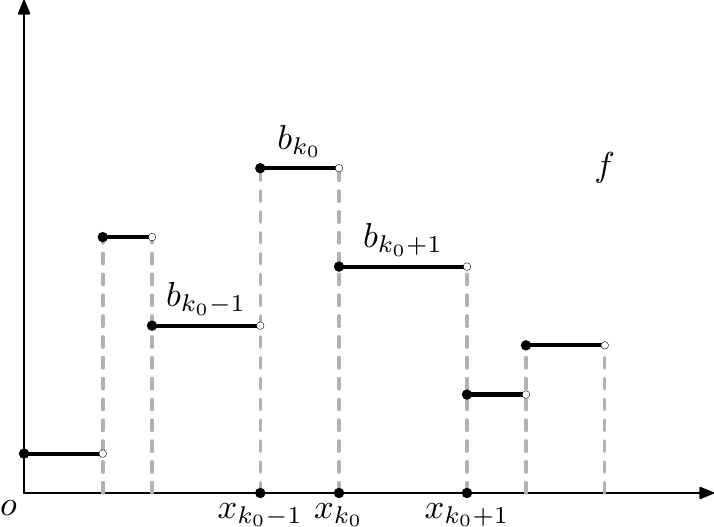}
}\\[0pt]
\subfigure[Step function $h$]{
    \includegraphics[width=0.43\textwidth]{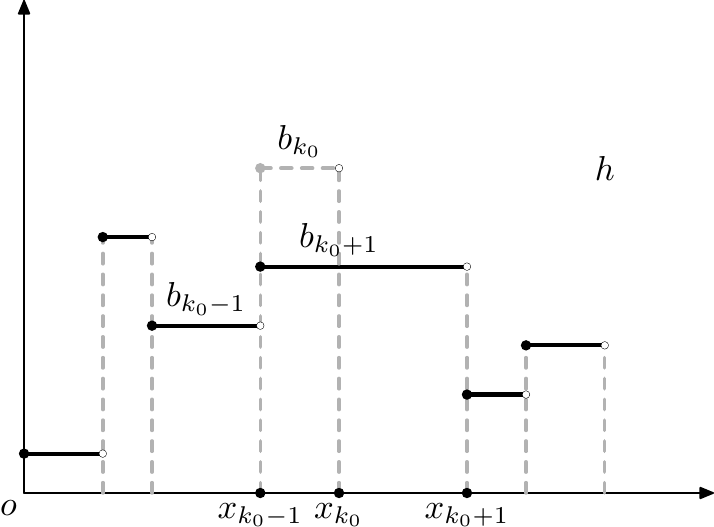}
} 
\subfigure[Step function $g$]{
    \includegraphics[width=0.43\textwidth]{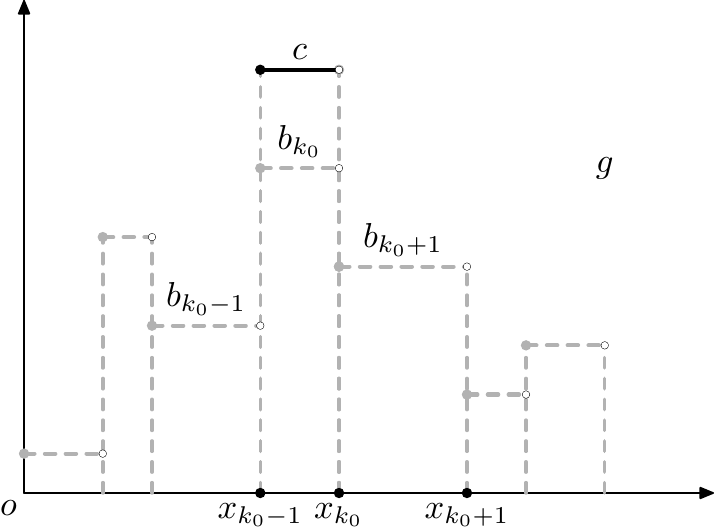}
} \label{Fig:step-funt-f}
\end{center}
\caption{Functions $f=pg+h$, $h$, and $g$}
\end{figure}

The function $h$ is constant on each interval $[x_{k-1},x_k)$. On $%
[x_{k_{0}-1},x_{k_{0}})$ and $[x_{k_{0}},x_{k_{0}+1})$, $h$ is equal to $%
b_{k_0+1}$. Therefore, by merging these two intervals, $h$ can be
represented as a step function, based on $n-1$ intervals, that is, 
\begin{equation*}
h=\sum_{k=1}^{k_0-1} b_k\Eins_{[x_{k-1},x_k)} + b_{k_0+1}\Eins%
_{[x_{k_{0}-1},x_{k_{0}+1})} + \sum_{k=k_0+2}^n b_k\Eins_{[x_{k-1},x_k)}.
\end{equation*}

By the induction hypothesis, there exist $n-1$ non-negative elementary step
functions $h_i$ and constants $q_i$, $i=1,2,\cdots,n-1$, such that 
\begin{equation}  \label{Equ:h=sum_qihi}
h=\sum_{i=1}^{n-1}q_{i}h_{i},
\end{equation}%
and 
\begin{equation}
\sum_{i=1}^{n-1}q_{i}=1,\quad q_{i}\geq 0,
\end{equation}

\begin{equation}  \label{Equ:nuh_i=nu_h}
\int _{\mathbb{R}_{+}}h_{i}\,d\nu =\int _{\mathbb{R}_{+}}h\,d\nu,
\end{equation}

\begin{equation}  \label{Equ:Vf=sum_qVh}
V_\nu(h) =\sum_{i=1}^{n-1}q_{i} V_\nu(h_{i}).
\end{equation}

It follows from (\ref{Equ:h=sum_qihi}) that 
\begin{equation*}
f=pg+\sum_{i=1}^{n-1}q_{i}h_{i}=pg+\sum_{i=1}^{n-1}q_{i}(1-p)\frac{h_{i}}{1-p%
}.
\end{equation*}%
Setting 
\begin{equation*}
f_{n}=g, \ p_{n}=p, \ f_{i}=\frac{h_{i}}{1-p} ,\ \ p_{i}=q_{i}(1-p)\ \ \text{%
for } i=1,2,\cdots ,n-1,
\end{equation*}
we obtain 
\begin{equation*}
f=\sum_{k=1}^{n}p_{k}f_{k}.
\end{equation*}
Moreover, we have 
\begin{equation*}
\sum_{k=1}^n p_k=p+\sum_{i=1}^{n-1}q_i (1-p)=p+(1-p)=1,
\end{equation*}
and 
\begin{equation*}
\int _{\mathbb{R}_{+}}f_n \,d\nu= \int _{\mathbb{R}_{+}}g \,d\nu= \int _{%
\mathbb{R}_{+}}f \,d\nu.
\end{equation*}
Since by (\ref{Equ:g-f-to-c}) 
\begin{equation*}
\int _{\mathbb{R}_{+}}h\,d\nu=\int _{\mathbb{R}_{+}}(f-pg)\,d\nu=(1-p)\int _{%
\mathbb{R}_{+}}f\,d\nu,
\end{equation*}
we obtain, for any $k=1,2,\cdots,n-1$, 
\begin{equation*}
\int _{\mathbb{R}_{+}}f_k \,d\nu = \int _{\mathbb{R}_{+}}\frac{h_{k}}{1-p}
\,d\nu =\frac{1}{1-p}\left(\int _{\mathbb{R}_{+}}h_{} \,d\nu \right) =\int _{%
\mathbb{R}_{+}}f\,d\nu.
\end{equation*}
By the construction of $h$ and $g$, at each point $x_{k}$ the jumps of $h$
and $g$ have the same sign as that of $f$, so that $V_{\mu}$ acts linearly
on the sum $f=h+pg$, consequently 
\begin{equation*}
V_\nu(f)=V_\nu(h) + p V_\nu(g).
\end{equation*}
By (\ref{Equ:Vf=sum_qVh}) we obtain 
\begin{equation*}
\begin{split}
V_\nu(f) &= \sum_{i=1}^{n-1}q_{i}V_\nu(h_i) + p V_\nu(g) \\
&=\sum_{i=1}^{n-1}\frac{p_i}{1-p}V_\nu(h_i) + p_n V_\nu(f_n) \\
&=\sum_{i=1}^{n-1}p_i V_\nu\left(\frac{1}{1-p}h_{i}\right) +p_{n}
V_\nu(f_{n}) \\
&=\sum_{i=1}^{n-1} p_i V_\nu(f_{i}) + p_n V_\nu(f_{n}) \\
&=\sum_{i=1}^{n} p_i V_\nu(f_{i}),
\end{split}%
\end{equation*}
which finishes the proof.
\end{proof}

\begin{corollary}
\label{Cor:2Lemma-combine}Under the hypotheses of Theorem \emph{\ref%
{Thm:iso-func-ineq}}, inequality \emph{(\ref%
{Equ:iso-func-ineq-1dim-general-more})} holds for all non-negative step
functions on $\mathbb{R}_+$.
\end{corollary}

\begin{proof}
By Lemma \ref{Lem:step-f-rewrite}, we can represent any non-negative step
function $f$ as the sum of non-negative elementary step functions such that
the conditions of Lemma \ref{Lem:f_k->step_f} are satisfied. Since for any
non-negative elementary function inequality (\ref%
{Equ:iso-func-ineq-1dim-general-more}) holds by Lemma \ref%
{Lem:func-ineq-for-1-step-f}, we conclude by Lemma \ref{Lem:f_k->step_f},
that $f$ satisfies (\ref{Equ:iso-func-ineq-1dim-general-more}).
\end{proof}

\begin{lemma}
\label{Lem:approx-by-step-f} Let $f$ be a non-negative continuously
differentiable function on $\mathbb{R}_{+}$ with support in an interval $%
[0,l].$ Consider the step function 
\begin{equation}
f_{n}=\sum_{k=1}^{n}f(x_{k})\Eins_{[x_{k-1},x_{k})},  \label{fn}
\end{equation}%
where $x_{k}=\frac{k}{n}l$. Then the sequence $\{f_{n}\}$ converges to $f$
as $n\rightarrow \infty $ uniformly on $\mathbb{R}_{+}$ and 
\begin{equation}
\lim_{n\rightarrow \infty }V_{\nu }(f_{n})=\int_{\mathbb{R}_{+}}|f^{\prime
}|\,d\nu .  \label{Equ:V_nuf_n-2-int-f'}
\end{equation}
\end{lemma}

\begin{proof}
The uniform convergence of $\{f_n\}$ to $f$ is obvious. We only need to show
(\ref{Equ:V_nuf_n-2-int-f'}). By the mean value theorem, for every $%
k=1,2,\cdots,n$, there exists some $\xi_k\in[x_{k},x_{k+1}]$ such that 
\begin{equation*}
f(x_{k+1})-f(x_{k})=f^{\prime }(\xi_k)(x_{k+1}-x_{k})= f^{\prime }(\xi_k) 
\frac{l}{n}.
\end{equation*}
It follows that 
\begin{equation}  \label{Equ:V_nufn=sum}
V_{\nu}(f_n)=\sum_{k=1}^{n} |f(x_{k+1})-f(x_{k})|\phi(x_{k}) =\sum_{k=1}^{n}
|f^{\prime }(\xi_k)|\phi(x_k) \frac{l}{n}.
\end{equation}
Since the function $|f^{\prime }|\phi$ is Riemann integrable, we have as $%
n\rightarrow\infty$ 
\begin{equation*}
\sum_{k=1}^{n} |f^{\prime }(x_k)|\phi(x_k) \frac{l}{n}\rightarrow\int_%
\mathbb{R_+}|f^{\prime }|\phi\, dx=\int_\mathbb{R_+}|f^{\prime }| d\nu.
\end{equation*}
On the other hand, we have 
\begin{equation*}
\left| \sum_{k=1}^{n} |f^{\prime }(\xi_k)|\phi(x_k) \frac{l}{n}%
-\sum_{k=1}^{n} |f^{\prime }(x_k)|\phi(x_k) \frac{l}{n}\right|\le
\sup_{k}|f^{\prime }(\xi_{k})-f^{\prime }(x_{k})|\sum_{k=1}^{n} \phi(x_k) 
\frac{l}{n}.
\end{equation*}
By the continuity of $f^{\prime }$, the $\sup$-term on the right hand side
tends to $0$ as $n\rightarrow\infty$. Since the sum-term tends to $%
\int^{l}_{0}\phi(x)dx<\infty,$ the whole expression tends to $0$, which
finishes the proof.
\end{proof}

\bigskip

\begin{proof}[Proof of Theorem \protect\ref{Thm:iso-func-ineq}]
Let $f$ be a non-negative continuously differentiable function on $\mathbb{R}%
_{+}$ with bounded support. Define $f_{n}$ by (\ref{fn}). By Corollary \ref%
{Cor:2Lemma-combine}, inequality (\ref{Equ:iso-func-ineq-1dim-general-more})
holds for each function $f_{n}$. Letting $n\rightarrow\infty$ by Lemma \ref%
{Lem:approx-by-step-f} we obtain that $f$ satisfies (\ref%
{Equ:iso-func-ineq-1dim-general}), which finishes the proof.
\end{proof}

\section{Functional isoperimetric inequality for the radial measure}

\label{Sec:ineq:radial}We here apply Theorem \ref{Thm:iso-func-ineq-special}
to obtain a functional isoperimetric inequality for the measure 
\begin{equation}
d\nu (r)=r^{n-1}\mathrm{e}^{-\frac{1}{r^{\alpha }}}dr  \label{nun}
\end{equation}%
on $(0,\infty )$, where $\alpha >0$ and $n\geq 1.$ Note that $\nu $ is the
radial part of the measure 
\begin{equation}
d\mu (x)=\mathrm{e}^{-\frac{1}{|x|^{\alpha }}}dx  \label{mun}
\end{equation}%
on $\mathbb{R}^{n}\setminus \{0\}$.

The isoperimetric function for the measure $\nu $ can be obtain from the
following result of \cite{BCM12}.

\begin{proposition}
\label{Prop:BCM01}(\cite[Proposition 3.1]{BCM12}) Let $\phi $ be a positive
continuous non-decreasing function defined on $\left( 0,+\infty \right) $.
Consider the Borel measure $d\nu (r)=\phi (r)\,dr$ on $\left( 0,+\infty
\right) $. Then for any Borel set $A\subset \left( 0,+\infty \right) $ we
have 
\begin{equation}
\nu ^{+}(A)\geq \nu ^{+}((0,r)),  \label{Equ:BCM01}
\end{equation}%
where $r\geq 0$ is chosen such that 
\begin{equation*}
\nu ((0,r))=\nu (A).
\end{equation*}%
Furthermore, if\,\ $\lim_{r\rightarrow 0}\phi (r)=0,$ then the isoperimetric
function $I_{\nu }$ is given by the identity 
\begin{equation*}
I_{\nu }(v)=\phi (r),
\end{equation*}%
where $v=\nu ((0,r))$.
\end{proposition}

Now we can determine a lower isoperimetric function for the measure defined
in (\ref{nun}).

\begin{proposition}
\label{Prop:J-def} There exists some constants $c,c^{\prime }>0$ and $%
0<v_{0}<1$ such that the function $J$, defined by 
\begin{equation}
J(v)=\left\{ 
\begin{array}{ll}
cv\left( \log \displaystyle\frac{1}{v}\right) ^{1+\frac{1}{\alpha }},\quad & 
0\leq v\leq v_{0}, \\ 
c^{\prime} \displaystyle v^{ \frac{n-1}{n}},\quad & v>v_{0},%
\end{array}%
\right.  \label{Equ:J(v)-def}
\end{equation}%
satisfies the following properties:

\begin{enumerate}
\item[$\left( i\right) $] $J$ is a lower isoperimetric function for the
measure $\nu $ given by \emph{(\ref{nun})}.

\item[$\left( ii\right) $] $J$ is concave, increasing and continuous on $%
\left( 0,+\infty \right) .$
\end{enumerate}
\end{proposition}

\begin{figure}[htpb]
\begin{center}
\includegraphics[width=0.7\textwidth]{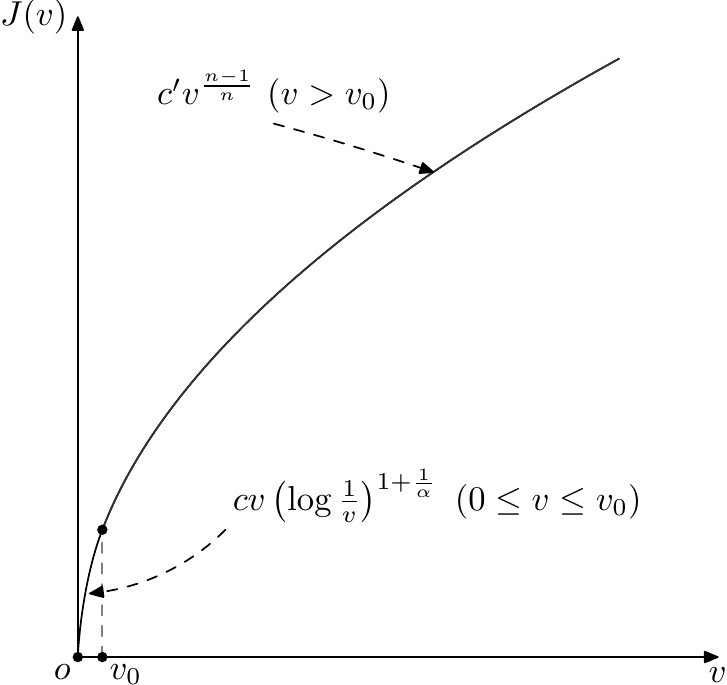}
\end{center}
\caption{Function $J$ defined by (\protect\ref{Equ:J(v)-def})}
\end{figure}

\begin{proof}
Since the function 
\begin{equation*}
\phi (r):=r^{n-1}\mathrm{e}^{-\frac{1}{r^{\alpha }}}
\end{equation*}%
is increasing in $r$, and $\lim_{r\rightarrow 0}\phi (r)=0$, by Proposition %
\ref{Prop:BCM01} we obtain 
\begin{equation*}
I_{\nu }(v)=\phi (R),
\end{equation*}%
where $R>0$ such that 
\begin{equation*}
v=\nu ((0,R))=\int_{0}^{R}\phi (r)\,dr=\int_{0}^{R}r^{n-1}\mathrm{e}^{-\frac{%
1}{r^{\alpha }}}\,dr.
\end{equation*}%
It is clear that, for large enough $R$, we have $\phi (R)\approx R^{n-1}$,
consequently $v\approx R^{n}$. Hence, for large enough $v$ we obtain 
\begin{equation}
I_{\nu }(v)\approx v^{\frac{n-1}{n}}.  \label{Equ:I_nu_v-large}
\end{equation}%
In order to estimate $v$ for small $R$, we shall use the following claim.

\begin{claim}
Let $F$ be a smooth enough positive function on $\left( 0,+\infty \right) $
such that%
\begin{equation}
a:=\lim_{x\rightarrow +\infty }\frac{F(x)F^{\prime \prime }(x)}{F^{\prime
2}\left( x\right) }>0\quad \text{and}\quad \int_{x}^{\infty }\frac{dr}{F(r)}%
<\infty .  \label{Equ:F''F'-condtion}
\end{equation}%
Then 
\begin{equation}
\int_{x}^{\infty }\frac{dr}{F(r)}\sim \frac{a^{-1}}{F^{\prime }(x)}\quad 
\text{as}\ x\rightarrow \infty .  \label{Equ:int-approx-formula}
\end{equation}
\end{claim}

Indeed, the estimate (\ref{Equ:int-approx-formula}) follows from
l'Hospital's rule since 
\begin{equation*}
\lim_{x\rightarrow +\infty }\frac{\int_{x}^{\infty }\frac{dr}{F(r)}}{\frac{1%
}{F^{\prime }(x)}}=\lim_{x\rightarrow +\infty }\frac{\frac{1}{F(x)}}{\frac{%
F^{\prime \prime }(x)}{F^{\prime 2}}}=\lim_{x\rightarrow +\infty }\frac{%
F^{\prime 2}(x)}{F(x)F^{\prime \prime }(x)}=a^{-1}.
\end{equation*}%
The function $F\left( x\right) =x^{n+1}\mathrm{e}^{x^{\alpha }}$ clearly
satisfies (\ref{Equ:F''F'-condtion}), and we obtain for small enough $R$%
\begin{equation}
\begin{split}
v =&\int_{0}^{R}r^{n-1}\mathrm{e}^{-r^{-\alpha }}\,dr=\int_{1/R}^{\infty }%
\frac{1}{x^{n-1}}\mathrm{e}^{-x^{\alpha }}\frac{1}{x^{2}}\,dx \\
=&\int_{1/R}^{\infty }\frac{dx}{x^{n+1}\mathrm{e}^{x^{\alpha }}}\approx
\left. \frac{1}{(x^{n+1}\mathrm{e}^{x^{\alpha }})^{\prime }}\right\vert _{x=%
\frac{1}{R}} \\
\approx &\, R^{n+\alpha }\mathrm{e}^{-\frac{1}{R^{\alpha }}}
\end{split}
\label{vap}
\end{equation}
It follows from (\ref{vap}) that 
\begin{equation}
\log v\approx (n+\alpha )\log R-\frac{1}{R^{\alpha }}\approx -\frac{1}{%
R^{\alpha }},  \label{Equ:v-approx-R}
\end{equation}%
consequently 
\begin{equation*}
\phi (R)=R^{n+\alpha }\mathrm{e}^{-\frac{1}{R^{\alpha }}}R^{-(1+\alpha
)}\approx vR^{-(1+\alpha )}\approx v\left( \log \frac{1}{v}\right) ^{1+\frac{%
1}{\alpha }}.
\end{equation*}%
Hence, for small enough $v$, we obtain%
\begin{equation}
I_{\nu }(v)\approx v\left( \log \frac{1}{v}\right) ^{1+\frac{1}{\alpha }}.
\label{Equ:I_nu_v-small}
\end{equation}%
Combining (\ref{Equ:I_nu_v-large}) and (\ref{Equ:I_nu_v-small}), we obtain
that the function $J$ from (\ref{Equ:J(v)-def}) is a lower isoperimetric
function for the measure $\nu $, for sufficiently small constants $v_{0}\in
(0,1)$ and $c,c^{\prime }>0$.

Consider the functions 
\begin{equation*}
I\left( v\right) =v\left( \log \frac{1}{v}\right) ^{1+\frac{1}{\alpha }}\ \ 
\text{and\ \ }I_{1}\left( v\right) =c_{1}v^{\frac{n-1}{n}}.
\end{equation*}%
Let us show that the constants $c_{1}>0$ and $v_{0}\in (0,1)$ can be chosen
so that the following function 
\begin{equation}
\tilde{J}(v):=\left\{ 
\begin{array}{ll}
I(v),\quad & 0\leq v\leq v_{0}, \\ 
I_{1}(v),\quad & v\geq v_{0}%
\end{array}%
\right.  \label{Equ:J(v)-tilde-def}
\end{equation}%
is concave, increasing and continuous on $\mathbb{R}_{+}$. Then the function 
$J=\func{const}\tilde{J}$ with small enough $\func{const}>0$ will satisfy 
both the conditions $\left( i\right) $ and $\left( ii\right) $.

The function $I_{1}\left( v\right) $ is clearly increasing and concave on $%
(0,+\infty )$. For the function $I\left( v\right) $ we have%
\begin{equation*}
\begin{split}
& I^{\prime }(v)=\left( \log \frac{1}{v}\right) ^{\frac{1}{\alpha }}\left(
\log \frac{1}{v}-\left( 1+\frac{1}{\alpha }\right) \right) , \\
& I^{\prime \prime }(v)=-\left( 1+\frac{1}{\alpha }\right) \frac{1}{v}\left(
\log \frac{1}{v}\right) ^{\frac{1}{\alpha }-1}\left( \log \frac{1}{v}-\frac{1%
}{\alpha }\right) ,
\end{split}%
\end{equation*}%
so that $I\left( v\right) $ is increasing and concave on the interval $(0,%
\mathrm{e}^{-(1+\frac{1}{\alpha })}).$ Now we choose $c_{1}>0$ and $%
v_{0}\leq \mathrm{e}^{-(1+\frac{1}{\alpha })}$ so that the function $\tilde{J%
}$ is of the class $C^{1}\left( 0,+\infty \right) $ and, hence, increasing
and concave on $\left( 0,+\infty \right) $. To that end, the following two
identities must be satisfied 
\begin{equation*}
\begin{split}
I\left( v_{0}\right) & =I_{1}\left( v_{0}\right) , \\
I^{\prime }\left( v_{0}\right) & =I_{1}^{\prime }\left( v_{0}\right) ,
\end{split}%
\end{equation*}%
which yields the following equations for $c_{1}$ and $v_{0}$:%
\begin{equation*}
\begin{split}
v_{0}\left( \log \frac{1}{v_{0}}\right) ^{1+\frac{1}{\alpha }}& =c_{1}v_{0}^{%
\frac{n-1}{n}}, \\
\left( \log \frac{1}{v_{0}}\right) ^{\frac{1}{\alpha }}\left( \log \frac{1}{%
v_{0}}-\left( 1+\frac{1}{\alpha }\right) \right) & =\frac{n-1}{n}%
c_{1}v_{0}^{-\frac{1}{n}}.
\end{split}%
\end{equation*}%
Multiplying the second equation by $v_{0}\log \frac{1}{v_{0}}$ and combining
this with the first, we obtain%
\begin{equation}
\log \frac{1}{v_{0}}-\left( 1+\frac{1}{\alpha }\right) =\frac{n-1}{n}\log 
\frac{1}{v_{0}},  \label{Equ:c1Iv0}
\end{equation}%
whence 
\begin{equation}
v_{0}=\mathrm{e}^{-n\left( 1+\frac{1}{\alpha }\right) }.  \label{v0}
\end{equation}%
The value of $c_{1}$ is then trivially determined from the one of the above
equations. The proof is finished by the observation that $v_{0}\leq \mathrm{e%
}^{-\left( 1+\frac{1}{\alpha }\right) }$.
\end{proof}

\begin{proposition}
\label{Prop:J-ineq} The function $J$ defined by \emph{(\ref{Equ:J(v)-def})}
satisfies the following property: there exists some constant $C_{J}>0$ such
that 
\begin{equation}
C_{J}J(ab)\leq bJ(a)+aJ(b)  \label{Equ:J(ab)-ineq}
\end{equation}%
for all $a,b\geq 0.$
\end{proposition}

\begin{proof}
If $a=0$ or $b=0$ then (\ref{Equ:J(ab)-ineq}) is trivial. Assume in the
sequel that $a,b>0$. Consider the function 
\begin{equation*}
F(v)=\frac{J(v)}{v}=\left\{ 
\begin{array}{ll}
c\left( \log \displaystyle \frac{1}{v}\right) ^{1+\frac{1}{\alpha }},\quad & 
0<v\leq v_{0}, \\ 
\displaystyle c^{\prime} v^{ -\frac{1}{n}},\quad & v\geq v_{0}.%
\end{array}%
\right.
\end{equation*}%
Obviously (\ref{Equ:J(ab)-ineq}) is equivalent to 
\begin{equation}
F(ab)\leq C_{J}^{-1}\left( F(a)+F(b)\right) ,  \label{Equ:J(ab)-ineq-F}
\end{equation}%
for all $a,b>0$. Without loss of generality, let us verify (\ref%
{Equ:J(ab)-ineq}) for $a\leq b$. We shall consider the following four cases.

\textbf{Case 1.} Assume that $b\geq 1$. Since $F$ is monotone decreasing and 
$ab\geq a$, we obtain 
\begin{equation}
F(ab)\leq F\left( a\right) \leq F(a)+F(b).  \label{Equ:Fab:case1}
\end{equation}%
In all the next cases we assume $b<1$.

\textbf{Case 2.} Assume that $a\leq v_{0}\leq b$. In this case we have $%
a^{2}\leq ab$ and, hence, 
\begin{equation}
F(ab)\leq F(a^{2}).  \label{Equ:Fab:case2-eq1}
\end{equation}%
Since $a^{2}<a<v_{0}$, we have 
\begin{equation}
F(a^{2})=c\left( \log \frac{1}{a^{2}}\right) ^{1+\frac{1}{\alpha }}=2^{1+%
\frac{1}{\alpha }}F(a).  \label{Equ:Fab:case2-eq2}
\end{equation}%
From (\ref{Equ:Fab:case2-eq1}) and (\ref{Equ:Fab:case2-eq2}) we obtain 
\begin{equation}
F(ab)\leq 2^{1+\frac{1}{\alpha }}F(a)\leq 2^{1+\frac{1}{\alpha }}\left(
F(a)+F(b)\right).  \label{Equ:Fab:case2}
\end{equation}

\textbf{Case 3.} Assume that $v_{0}\leq a\leq b$. In this case we have $%
ab\geq v_{0}^{2}$ and, hence, 
\begin{equation}
F(ab)\leq F(v_{0}^{2}).  \label{Equ:Fab:case3-eq1}
\end{equation}%
On the other hand, since $a,b<1$, we have 
\begin{equation*}
F(a)+F\left( b\right) \geq F(1)+F\left( 1\right) =2F(1).
\end{equation*}%
Combining this with (\ref{Equ:Fab:case3-eq1}) we obtain 
\begin{equation}
F(ab)\leq \frac{F(v_{0}^{2})}{2F(1)}\left( F(a)+F(b)\right) .
\label{Equ:Fab:case3}
\end{equation}

\textbf{Case 4 (main).} Assume that $a\leq b\leq v_{0}$. Since $ab<v_{0}$,
we obtain

\begin{equation}
\begin{split}
F(ab) &=c\left( \log \frac{1}{ab}\right) ^{1+\frac{1}{\alpha }}=c\left( \log 
\frac{1}{a}+\log \frac{1}{b}\right) ^{1+\frac{1}{\alpha }} \\
&\leq 2^{\frac{1}{\alpha }}c\left( \left( \log \frac{1}{a}\right) ^{1+\frac{1%
}{\alpha }}+\left( \log \frac{1}{b}\right) ^{1+\frac{1}{\alpha }}\right) \\
&= 2^{\frac{1}{\alpha }}\left( F\left( a\right) +F\left( b\right) \right)
\label{Equ:Fab:case4}
\end{split}%
\end{equation}

Combining (\ref{Equ:Fab:case1}), (\ref{Equ:Fab:case2}), (\ref{Equ:Fab:case3}%
) and (\ref{Equ:Fab:case4}) we obtain (\ref{Equ:J(ab)-ineq-F}) and hence (%
\ref{Equ:J(ab)-ineq}) with 
\begin{equation}
C_{J}=\min \left( 2^{-\left( 1+\frac{1}{\alpha }\right) },\frac{2F(1)}{%
F(v_{0}^{2})}\right) ,  \label{Equ:C_J-constant}
\end{equation}%
which finishes the proof.
\end{proof}

\bigskip

By Theorem \ref{Thm:iso-func-ineq} and Propositions \ref{Prop:J-def}, \ref%
{Prop:J-ineq} we obtain the following result.

\begin{theorem}
\label{Thm:iso-func-ineq:radial} The function $J$ given by \emph{(\ref%
{Equ:J(v)-def})} is a lower isoperimetric function for the measure $d\nu
(r)=r^{n-1}\mathrm{e}^{-\frac{1}{r^{\alpha }}}dr$ on $\mathbb{R}_{+}$.
Moreover, for any non-negative continuously differentiable function $f$ on $%
\mathbb{R}_{+}$ with bounded support we have 
\begin{equation}
C_{J}J\left( \int_{\mathbb{R}_{+}}f\,d\nu \right) \leq \int_{\mathbb{R}%
_{+}}J(f)\,d\nu +\int_{\mathbb{R}_{+}}|f^{\prime }|d\nu ,
\label{Ineq:func-iso-radial}
\end{equation}%
where $C_{J}$ is the constant from \emph{(\ref{Equ:J(ab)-ineq})}.
\end{theorem}

\section{Functional isoperimetric inequality on a sphere}

\label{Sec:ineq:angular} We shall use the following result of \cite{Bar02}
about isoperimetric inequalities for probability measures that we state here
in a specific setting adapted to our needs.

\begin{theorem}
\label{Lmm:Bar02Thm2}\emph{(\cite[Theorem 2]{Bar02})} Let $L$ be a
non-negative function on $\left[ 0,1\right] $ with the following properties:

\begin{enumerate}
\item[$\left( i\right) $] $L$ is continuous, concave and symmetric with
respect to $1/2$, and $L(0)=L\left( 1\right) =0.$

\item[$\left( ii\right) $] For some constant $C_{L}>0$ and for all $a,b\in
\lbrack 0,1]$, 
\begin{equation}
C_{L}L(ab)\leq bL(a)+aL(b).  \label{Equ:Bar02Thm2-L-ineq:Properties}
\end{equation}
\end{enumerate}

Let $\left( N,\sigma \right) $ be a weighted manifold and $\sigma \left(
N\right) =1$. If $L$ is a lower isoperimetric function for the measure $%
\sigma $, then, for any locally Lipschitz function $f\colon N\rightarrow
\lbrack 0,1] $, we have 
\begin{equation}
C_{L}L\left( \int_{N}f\,d\sigma \right) \leq \int_{N}L(f)\,d\sigma
+\int_{N}|\nabla f|\,d\sigma .  \label{Equ:Bar02Thm2-L-ineq:func}
\end{equation}
\end{theorem}

Let $\sigma_{n-1}$ denote the canonical spherical measure on $\mathbb{S}%
^{n-1}$. Set $\omega _{n}=\sigma _{n-1}\left( \mathbb{S}^{n-1}\right) $ and
consider the normalized spherical measure 
\begin{equation*}
\tilde{\sigma}_{n-1}=\frac{1}{\omega_{n}}\sigma _{n-1}.
\end{equation*}%
Before we apply Theorem \ref{Lmm:Bar02Thm2} to $(\mathbb{S}^{n-1},\tilde{%
\sigma}_{n-1})$, we need to construct a function $L$ satisfying appropriate
conditions.

\begin{proposition}
\label{Prop:L-def-properties} Choose some $\beta >1$ and $n\geq 1$, set $%
v_{0}=\mathrm{e}^{-n\beta }$ and consider the functions $I$ and $L$ on $%
\left[ 0,1\right] $ defined by%
\begin{equation}
I\left( v\right) =v\left( \log \frac{1}{v}\right) ^{\beta }  \label{Idef}
\end{equation}%
and%
\begin{equation}
L\left( v\right) =c\left\{ 
\begin{array}{ll}
I\left( v\right) , & 0\leq v\leq v_{0}, \\ 
I\left( v_{0}\right) , & v_{0}< v< 1-v_{0}, \\ 
I\left( 1-v\right) , & 1-v_{0}\leq v\leq 1,%
\end{array}%
\right.  \label{Equ:L-def}
\end{equation}%
where $c$ is a positive constant. Then $L$ satisfies the following
properties:

\begin{enumerate}
\item[$\left( i\right) $] $L$ is continuous, concave and symmetric with
respect to $1/2$.

\item[$\left( ii\right) $] If $c$ is sufficiently small, then $L$ is a lower
isoperimetric function for the measure $\sigma _{n-1}$ on $\mathbb{S}^{n-1}$.

\item[$\left( iii\right) $] There exists a constant $C_{L}>0$ such that 
\begin{equation}
C_{L}L(ab)\leq bL(a)+aL(b)  \label{Equ:Ineq-Lab}
\end{equation}%
for all $0<a,b<1$.
\end{enumerate}
\end{proposition}

\begin{figure}[tph]
\begin{center}
\includegraphics[width=0.8\textwidth]{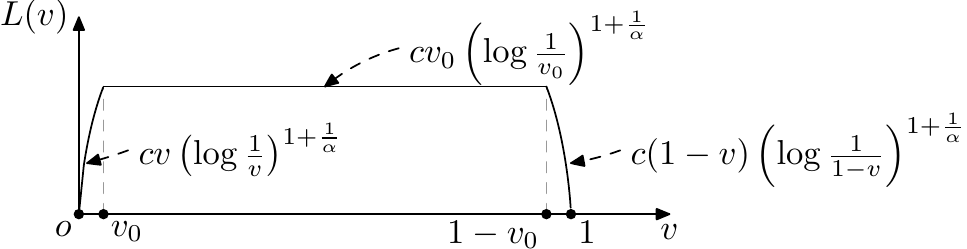}
\end{center}
\caption{Function $L$ defined by (\protect\ref{Equ:L-def})}
\end{figure}

\begin{remark}
\RM We shall use Proposition \ref{Prop:L-def-properties} with $\beta =1+%
\frac{1}{\alpha }$, where $\alpha $ is the constant in the definitions (\ref%
{nun}) and (\ref{mu}) of the measures $\nu $ and $\mu $, respectively. By
Theorem \ref{Thm:iso-func-ineq:radial} we have a lower isoperimetric
function $J$ for the measure $\mu $ that is given by (\ref{Equ:J(v)-def}).
In the next section we shall combine the isoperimetric functions $J$ and $L$
in order to obtain a lower isoperimetric function of the measure $\mu $.
Note that the parameter $v_{0}$ in (\ref{Equ:J(v)-def}) and (\ref{Equ:L-def}%
) has the same value given by (\ref{v0}). It will be convenient to assume
that the constants $c$ in (\ref{Equ:J(v)-def}) and (\ref{Equ:L-def}) also
have the same value, which can always be achieved. Hence, we have 
\begin{equation}
J(v)=L(v)=cI\left( v\right) ,\quad \text{for\ all}\ 0\leq v\leq v_{0}.
\label{JLI}
\end{equation}
\end{remark}

\begin{proof}
$(i)$ From (\ref{Equ:L-def}) it is clear that $L$ is continuous and
symmetric. The concavity follows from Proposition \ref{Prop:J-def}.

$(ii)$ Set 
\begin{equation*}
I_{\mathbb{S}^{n-1}}(v)=\left\{ 
\begin{array}{ll}
\displaystyle c_{n}v^{\frac{n-2}{n-1}},\quad & 0\leq v\leq 1/2, \\ 
\displaystyle c_{n}(1-v)^{\frac{n-2}{n-1}},\quad & 1/2<v\leq 1,%
\end{array}%
\right.
\end{equation*}%
where $c_{n}>0$ is a constant. It is well known that $I_{\mathbb{S}^{n-1}}$
is a lower isoperimetric function on $\mathbb{S}^{n-1}$ with respect to $%
\sigma _{n-1}$, provided $c_{n}$ is sufficiently small.

If $c>0$ is sufficiently small then we have for all $v\in (0,\frac{1}{2})$ 
\begin{equation*}
cv\left( \log \frac{1}{v}\right) ^{\beta }\leq c_{n}v^{\frac{n-2}{n-1}}
\end{equation*}%
and, hence, $L\left( v\right) \leq I_{\mathbb{S}^{n-1}}\left( v\right) $ for
all $v\in \left( 0,1\right) $. Consequently, $L$ is a lower isoperimetric
function.

$(iii)$ If $a$ or $b$ are equal to $0$ or $1$, then (\ref{Equ:Ineq-Lab}) is
trivially satisfied, so we can assume in the sequel $a,b\in \left(
0,1\right).$ Define $F\colon (0,1)\to \mathbb{R} $ by%
\begin{equation}
F\left( v\right) =\left\{ 
\begin{array}{ll}
\displaystyle\left( \log \frac{1}{v}\right) ^{\beta }, & 0<v\leq v_{0}, \\ 
\displaystyle\left( \log \frac{1}{v_{0}}\right) ^{\beta }, & v_{0}\leq v\leq
1-v_{0}, \\ 
\displaystyle \frac{1-v}{v_{0}}\left( \log \frac{1}{1-v}\right) ^{\beta }, & 
1-v_{0}\leq v<1.%
\end{array}%
\right.  \label{Equ:L-def1}
\end{equation}%
Then $F$ is positive, continuous and decreasing on $\left( 0,1\right) $, and 
\begin{equation*}
\frac{L\left( v\right) }{v}\approx F\left( v\right) \ \ \text{for }v\in
\left( 0,1\right) .
\end{equation*}%
Hence, (\ref{Equ:Ineq-Lab}) is equivalent to 
\begin{equation}
F(ab)\leq \func{const}\left( F(a)+F(b)\right) ,  \label{Equ:Ineq-L-Fab}
\end{equation}%
for all $a,b\in \left( 0,1\right) .$

Without loss of generality, we can assume that $a\leq b.$ Since the quotient 
$\frac{F\left( ab\right) }{F\left( a\right) +F\left( b\right) }$ can blow up
only on the boundary of the square $\left( 0,1\right) ^{2}$, it suffices to
prove (\ref{Equ:Ineq-L-Fab}) in two cases: $a\leq v_{0}$ and $b\geq 1-v_{0}.$

\textbf{Case} $a\leq v_{0}.$ Then also $ab\leq v_{0}$, and we obtain by (\ref%
{Equ:L-def1}) 
\begin{equation*}
\begin{split}
F(ab) &=\left( \log \frac{1}{ab}\right) ^{\beta }=\left( \log \frac{1}{a}%
+\log \frac{1}{b}\right) ^{\beta } \\
&\leq 2^{\beta -1}\left( \left( \log \frac{1}{a}\right) ^{\beta }+\left(
\log \frac{1}{b}\right) ^{\beta }\right) \\
&= 2^{\beta -1}F\left( a\right) +2^{\beta -1}\left( \log \frac{1}{b}\right)
^{\beta }.
\end{split}%
\end{equation*}
We are left to show that, for all $b\in \left( 0,1\right) $, 
\begin{equation*}
\left( \log \frac{1}{b}\right) ^{\beta }\leq \func{const}F\left( b\right).
\end{equation*}%
For $b<v_{0}$ this is an identity, and the case $b\geq v_{0}$ will follow if
we verify that%
\begin{equation*}
\lim_{v\rightarrow 1^{-}}\frac{\left( \log \frac{1}{v}\right) ^{\beta }}{%
F\left( v\right) }<\infty,
\end{equation*}%
which amounts to showing that%
\begin{equation*}
\lim_{v\rightarrow 1^{-}}\frac{\left( \log \frac{1}{v}\right) ^{\beta }}{%
\left( 1-v\right) \left( \log \frac{1}{1-v}\right) ^{\beta }}<\infty .
\end{equation*}%
Since $\log \frac{1}{v}\sim 1-v\ $as $v\rightarrow 1^{-}$, we see that the
above limit is equal to 
\begin{equation*}
\lim_{v\rightarrow 1^{-}}\frac{\left( 1-v\right) ^{\beta -1}}{\left( \log 
\frac{1}{1-v}\right) ^{\beta }}=0,
\end{equation*}%
which finishes the proof in this case.

\textbf{Case} $b\geq 1-v_{0}.$ We can assume that $a\geq v_{0}.$ Consider
the function 
\begin{equation*}
G\left( x\right) =F\left( 1-x\right),\quad x\in(0,1),
\end{equation*}%
and restate (\ref{Equ:Ineq-L-Fab}) as follows:%
\begin{equation*}
G\left( \left( x+y\right) -xy\right) \leq \func{const}\left( G\left(
x\right) +G\left( y\right) \right)
\end{equation*}%
for all $y\leq v_{0}$ and $x\leq 1-v_{0}$.\ Since $G$ is increasing, it
suffices to prove that 
\begin{equation*}
G\left( x+y\right) \leq \func{const}\left( G\left( x\right) +G\left(
y\right) \right) ,
\end{equation*}%
and the latter inequality is obvious, since%
\begin{equation*}
G\left( x\right) =\frac{1}{v_{0}}\left\{ 
\begin{array}{ll}
\displaystyle x\left( \log \frac{1}{x}\right) ^{\beta }, & x\leq v_{0}, \\ 
\displaystyle v_{0}\left( \log \frac{1}{v_{0}}\right) ^{\beta }, & v_{0}\leq
x\leq 1-v_{0}.%
\end{array}%
\right.
\end{equation*}
{\hfill }
\end{proof}

Applying Theorem \ref{Lmm:Bar02Thm2} we obtain the following result.

\begin{theorem}
\label{Thm:func-ineq-iso-sphere} Let $L$ be defined as in \emph{(\ref%
{Equ:L-def})}. Then any $C^{1}$ function $f\colon \mathbb{S}%
^{n-1}\rightarrow \lbrack 0,1]$ satisfies the following inequality 
\begin{equation}
\omega _{n}C_{L}L\left( \frac{1}{\omega _{n}}\int_{\mathbb{S}%
^{n-1}}f\,d\sigma _{n-1}\right) \leq \int_{\mathbb{S}^{n-1}}L(f)\,d\sigma
_{n-1}+\int_{\mathbb{S}^{n-1}}|\nabla f|\,d\sigma _{n-1},
\label{Equ:func-ineq-iso-sphere:Thm}
\end{equation}%
where $\omega _{n}=\sigma _{n-1}\left( \mathbb{S}^{n-1}\right) $ and $C_{L}$
is the constant from \emph{(\ref{Equ:Ineq-Lab})}.
\end{theorem}

\begin{proof}
(\ref{Equ:func-ineq-iso-sphere:Thm}) is a direct consequence of the
following inequality 
\begin{equation*}
C_{L}L\left( \int_{\mathbb{S}^{n-1}}f\,d\tilde{\sigma}_{n-1}\right) \leq
\int_{\mathbb{S}^{n-1}}L(f)\,d\tilde{\sigma}_{n-1}+\int_{\mathbb{S}%
^{n-1}}|\nabla f|\,d\tilde{\sigma}_{n-1},
\end{equation*}%
that in turn follows from Theorem \ref{Lmm:Bar02Thm2} and the properties of $%
L$ stated in Proposition \ref{Prop:L-def-properties}.
\end{proof}

\section{Isoperimetric inequality for a weighted measure on $\mathbb{R}%
^{n}\setminus \{0\}$}

\label{Sec:ineq:combine} In this section we again consider the measure 
\begin{equation*}
d\mu (x)=\mathrm{e}^{-\frac{1}{|x|^{\alpha }}}dx
\end{equation*}%
on $M:=\mathbb{R}^{n}\setminus \{0\}$, where $\alpha >0$. Consider also the
radial part of $\mu $, that is, the measure $\nu $ on $\mathbb{R}_{+}$ given
by 
\begin{equation*}
d\nu (r)=r^{n-1}\mathrm{e}^{-\frac{1}{r^{\alpha }}}dr.
\end{equation*}%
For any $R>0$, set%
\begin{equation*}
B_{R}:=\{x\in \mathbb{R}^{n}\setminus \{0\}\colon \left\vert x\right\vert
<R\}.
\end{equation*}%
Let $\bar{B}_{R}$ denote the closure of $B_{R}$ in $\mathbb{R}^{n}$, i.e. 
\begin{equation*}
\bar{B}_{R}:=\{x\in \mathbb{R}^{n}\colon \left\vert x\right\vert \leq R\}.
\end{equation*}

\begin{theorem}
\label{Thm:iso-ineq-Rn} Let $f$ be a $C^{1}$ function on $M$ with support in 
$\bar{B}_{R}$ for some $R>0$. Assume that 
\begin{equation}
0\leq f\leq \frac{v_{0}}{\nu ((0,R))}\wedge v_{0},
\label{Equ:Thm:iso-ineq-Rn-1}
\end{equation}%
where $v_{0}=\mathrm{e}^{-n\left( 1+\frac{1}{\alpha }\right) }$ (cf. \emph{(%
\ref{v0})}). Then 
\begin{equation}
\omega _{n}C_{J}C_{L}I\left( \frac{1}{\omega _{n}}\int_{M}f\,d\mu \right)
\leq \int_{M}I(f)\,d\mu +\frac{1}{c}\left( 1+C_{J}R\right) \int_{M}|\nabla
f|\,d\mu ,  \label{Equ:Thm:iso-ineq-Rn-2}
\end{equation}%
where 
\begin{equation*}
I(v)=v\left( \log \frac{1}{v}\right) ^{1+\frac{1}{^{\alpha }}},
\end{equation*}%
$C_{J},C_{L}$ are the constants from Theorems \emph{\ref%
{Thm:iso-func-ineq:radial}} and \emph{\ref{Thm:func-ineq-iso-sphere}}
respectively, and $c$ is the constant from \emph{(\ref{JLI})}.
\end{theorem}

\begin{proof}
Let us use polar coordinates $\left( r,\theta \right) $ in $M=\mathbb{R}%
^{n}\setminus \left\{ 0\right\} $, where $r>0$ is the polar radius and $%
\theta \in \mathbb{S}^{n-1}$ is the polar angle (that is, for any $x\in M$
we have $r=\left\vert x\right\vert $ and $\theta =\frac{x}{\left\vert
x\right\vert }$). Let $f$ be a $C^{1}$ function on $M$ with support in $\bar{%
B}_{R}$ that satisfies (\ref{Equ:Thm:iso-ineq-Rn-1}). Consider the following
function $F$ on $\mathbb{S}^{n-1}$:%
\begin{equation*}
F\left( \theta \right) =\int_{\mathbb{R}_{+}}f\left( r,\theta \right) d\nu
\left( r\right) .
\end{equation*}%
By (\ref{Equ:Thm:iso-ineq-Rn-1}) we have%
\begin{equation}
0\leq F\leq v_{0}  \label{Fv0}
\end{equation}%
and, consequently, 
\begin{equation}
0\leq \frac{1}{\omega _{n}}\int_{\mathbb{S}^{n-1}}F(\theta )\,d\sigma
_{n-1}(\theta )\leq v_{0}.  \label{Equ:func-ineq-iso-sphere:e2}
\end{equation}%
Applying the estimate (\ref{Equ:func-ineq-iso-sphere:Thm}) of Theorem \ref%
{Thm:func-ineq-iso-sphere} to $F$ and noting that the function $L$ on the
range of $F$ can be replaced by $J$ or $cI$ (cf. (\ref{JLI})),  we obtain 
\begin{equation}
\omega _{n}C_{L}cI\left( \frac{1}{\omega _{n}}\int_{\mathbb{S}%
^{n-1}}F\,d\sigma _{n-1}\right) \leq \int_{\mathbb{S}^{n-1}}J(F)\,d\sigma
_{n-1}+\int_{\mathbb{S}^{n-1}}|\nabla _{\theta }F|\,d\sigma _{n-1}.
\label{Equ:func-ineq-iso-sphere:2F}
\end{equation}%
For the term in the left hand side we have 
\begin{equation}
\int_{\mathbb{S}^{n-1}}Fd\sigma _{n-1}=\int_{\mathbb{S}^{n-1}}\int_{\mathbb{R%
}_{+}}f\,d\nu d\sigma _{n-1}=\int_{M}f\,d\mu .
\label{Equ:func-ineq-iso-sphere:e1}
\end{equation}%
For the right hand side of (\ref{Equ:func-ineq-iso-sphere:2F}), we apply
Theorem \ref{Thm:iso-func-ineq:radial} to the function $f\left( \theta
,\cdot \right) $ and obtain 
\begin{equation}
\begin{split}
C_{J}J(F(\theta ))& =C_{J}J\left( \int_{\mathbb{R}_{+}}f(r,\theta )\,d\nu
(r)\right) \\
& \leq \int_{\mathbb{R}_{+}}J(f)\,d\nu +\int_{\mathbb{R}_{+}}|f_{r}|\,d\nu \\
& =\int_{\mathbb{R}_{+}}cI(f)\,d\nu +\int_{\mathbb{R}_{+}}|f_{r}|\,d\nu ,
\end{split}
\label{Equ:func-ineq-iso-sphere:e4}
\end{equation}%
where we have used that $J\left( f\right) =cI\left( f\right) $, which in
turn is true by (\ref{JLI}), because $0\leq f\leq v_{0}$. Combining (\ref%
{Equ:func-ineq-iso-sphere:2F}), (\ref{Equ:func-ineq-iso-sphere:e4}), and
using that%
\begin{equation*}
\left\vert \nabla _{\theta }F\right\vert \leq \int_{\mathbb{R}%
_{+}}\left\vert \nabla_{\theta} f \right\vert d\nu ,
\end{equation*}%
we obtain 
\begin{equation}
\begin{split}
& \omega _{n}C_{L}C_{J}cI\left( \frac{1}{\omega _{n}}\int_{M}f\,d\mu \right)
\\
\leq & \int_{\mathbb{S}^{n-1}}\int_{\mathbb{R}_{+}}cI(f)\,d\nu \,d\sigma
_{n-1}+\int_{\mathbb{S}^{n-1}}\int_{\mathbb{R}_{+}}|f_{r}|\,d\nu \,d\sigma
_{n-1}+C_{J}\int_{\mathbb{S}^{n-1}}\int_{\mathbb{R}_{+}}|\nabla _{\theta
}f|d\nu \,d\sigma _{n-1}
\end{split}
\label{Equ:func-ineq-iso-sphere:e5}
\end{equation}%
Note that%
\begin{equation*}
\left\vert \nabla f\right\vert ^{2}=f_{r}^{2}+\frac{1}{r^{2}}\left\vert
\nabla _{\theta }f\right\vert ^{2},
\end{equation*}%
whence%
\begin{equation*}
\left\vert f_{r}\right\vert +C_{J}\left\vert \nabla _{\theta }f\right\vert
\leq \left\vert \nabla f\right\vert +C_{J}r\left\vert \nabla f\right\vert .
\end{equation*}%
Since $f$ is supported in $\bar{B}_{R}$, the value of the polar radius $r$
in the integrals of (\ref{Equ:func-ineq-iso-sphere:e5}) is bounded by $R$.
Hence,%
\begin{equation*}
\left\vert f_{r}\right\vert +C_{J}\left\vert \nabla _{\theta }f\right\vert
\leq \left( 1+C_{J}R\right) \left\vert \nabla f\right\vert,
\end{equation*}%
whence we obtain%
\begin{equation*}
\omega _{n}C_{L}C_{J}cI\left( \frac{1}{\omega _{n}}\int_{M}f\,d\mu \right)
\leq c\int_{M}I\left( f\right) d\mu +\left( 1+C_{J}R\right)
\int_{M}\left\vert \nabla f\right\vert d\mu.
\end{equation*}%
Dividing both sides by $c$, we obtain (\ref{Equ:Thm:iso-ineq-Rn-2}).
\end{proof}

\bigskip

Now we shall apply the functional isoperimetric inequality (\ref%
{Equ:Thm:iso-ineq-Rn-2}) in order to prove an isoperimetric inequality for
the measure $\mu .$ In the next statement we first obtain a restricted
version of the isoperimetric inequality. We use the same notation as above.

\begin{lemma}
\label{Thm:iso-ineq-small-volume-ball} There are constants $R>0$ and $C>0$
such that, for any Borel set $A\subset B_{R}$, 
\begin{equation}
\mu ^{+}(A)\geq CI(\mu (A)).  \label{Equ:iso-ineq-small-volume-ball}
\end{equation}
\end{lemma}

\begin{proof}
For any $\varepsilon >0$, let $f_{\varepsilon }$ be a smooth approximation
of $\frac{v_{0}}{2}\Eins_{A}$ such that

\begin{enumerate}
\item[$\left( a\right) $] $f_{\varepsilon }=0$ outside $A^{\varepsilon }$;

\item[$\left( b\right) $] $0\leq f_{\varepsilon }\leq v_{0}.$
\end{enumerate}

We can assume that the value of $R$ that we seek is small enough, so that%
\begin{equation}
\nu \left( \left( 0,R\right) \right) \leq v_{0}.  \label{R0}
\end{equation}%
Then $f_{\varepsilon }$ satisfies (\ref{Equ:Thm:iso-ineq-Rn-1}), and we
obtain by (\ref{Equ:Thm:iso-ineq-Rn-2})%
\begin{equation}
\omega _{n}C_{J}C_{L}I\left( \frac{1}{\omega _{n}}\int_{M}f_{\varepsilon
}\,d\mu \right) \leq \int_{M}I(f_{\varepsilon })\,d\mu +\frac{1}{c}\left(
C_{J}R+1\right) \int_{M}|\nabla f_{\varepsilon }|\,d\mu .
\label{Equ:Thm:iso-ineq-Rn-2-apply1}
\end{equation}%
Letting $\varepsilon \rightarrow 0$ in (\ref{Equ:Thm:iso-ineq-Rn-2-apply1})
we obtain 
\begin{equation}
\omega _{n}C_{J}C_{L}I\left( \frac{v_{0}}{2\omega _{n}}\mu (A)\right) \leq
I\left( \frac{v_{0}}{2}\right) \mu (A)+\frac{v_{0}}{2c}\left(
C_{J}R+1\right) \mu ^{+}(A).  \label{Equ:Thm:iso-ineq-Rn-2-apply2}
\end{equation}%
Let us show that if $R$ is small enough, then 
\begin{equation}
I\left( \frac{v_{0}}{2}\right) \mu (A)\leq \frac{1}{2}\omega
_{n}C_{J}C_{L}I\left( \frac{v_{0}}{2\omega _{n}}\mu (A)\right) .
\label{I1/2I}
\end{equation}%
Indeed, using $I\left( v\right) =v\left( \log \frac{1}{v}\right) ^{\beta }$
where $\beta =1+\frac{1}{\alpha }$, we obtain that (\ref{I1/2I}) is
equivalent to 
\begin{equation*}
\left( \log \frac{2}{v_{0}}\right) ^{\beta }\leq \frac{1}{2}C_{J}C_{L}\left(
\log \frac{2\omega _{n}}{v_{0}\mu \left( A\right) }\right) ^{\beta },
\end{equation*}%
which in turn is equivalent to%
\begin{equation*}
\mu \left( A\right) \leq \omega _{n}\left( \frac{v_{0}}{2}\right) ^{N},
\end{equation*}%
where $N=\left( \frac{1}{2}C_{J}C_{L}\right) ^{-1/\beta }-1.$ Since $\mu
\left( A\right) \leq \mu \left( B_{R}\right) $, this inequality will be 
satisfied provided%
\begin{equation}
\mu \left( B_{R}\right) \leq \omega _{n}\left( \frac{v_{0}}{2}\right) ^{N}.
\label{RN}
\end{equation}%
Hence, for the value of $R$ that satisfies both (\ref{R0}) and (\ref{RN}),
we obtain%
\begin{equation*}
\frac{1}{2}\omega _{n}C_{J}C_{L}I\left( \frac{v_{0}}{2\omega _{n}}\mu
(A)\right) \leq \frac{v_{0}}{2c}\left( C_{J}R+1\right) \mu ^{+}(A),
\end{equation*}%
whence (\ref{Equ:iso-ineq-small-volume-ball}) follows.
\end{proof}

\bigskip

Now we are ready to prove a full isoperimetric inequality for $\mu $. This
is the main technical result of this paper.

\begin{theorem}
\label{Thm:iso-ineq-global}There exist constants $C>0$ and $\tau \in \left(
0,1\right) $ such that the following function 
\begin{equation}
\tilde{I}(v)=C\left\{ 
\begin{array}{ll}
v\left( \log \frac{1}{v}\right) ^{1+\frac{1}{\alpha }},\quad & 0\leq v\leq
\tau , \\ 
v^{\frac{n-1}{n}},\quad & v>\tau%
\end{array}%
\right.  \label{Equ:iso-ineq-global}
\end{equation}%
is a lower isoperimetric function for the measure $\mu $ on $M$.
\end{theorem}

\begin{proof}
We shall use the function $I\left( v\right) =v\left( \log \frac{1}{v}\right)
^{1+\frac{1}{\alpha }}$ as before. By Theorem \ref%
{Thm:iso-ineq-small-volume-ball}, there exist some $R>0$ and a constant $%
C_{0}>0$ such that for all Borel sets $A\subset B_{R}$ 
\begin{equation}
\mu ^{+}(A)\geq C_{0}I\left( \mu \left( A\right) \right) .
\label{Equ:iso-ineq-global-1}
\end{equation}%
Since for all $|x|>R$ we have $\mathrm{e}^{-\frac{1}{|x|^{\alpha }}}\approx
1 $, the measure $\nu $ outside $B_{R}$ is in finite ratio with Lebesgue
measure, which implies that for all Borel sets $A\subset
B_{R}^{c}:=M\setminus B_{R}$, 
\begin{equation}
\mu ^{+}(A)\geq C_{1}\left( \mu (A)\right) ^{\frac{n-1}{n}},
\label{Equ:iso-ineq-global-2}
\end{equation}%
for some constant $C_{1}>0$.

For any Borel set $\Omega \subset M$, setting 
\begin{equation*}
\Omega _{0}=B_{R}\cap \Omega ,\quad \Omega _{1}=B_{R}^{c}\cap \Omega,
\end{equation*}%
let us prove that 
\begin{equation}
3\mu ^{+}(\Omega )\geq C_{0}I\left( \mu (\Omega _{0})\right) +C_{1}\mu
(\Omega _{1})^{\frac{n-1}{n}}.  \label{Equ:iso-ineq-global-5}
\end{equation}%
Set 
\begin{equation*}
\Gamma _{0}=\partial \Omega \cap B_{R},\quad \Gamma _{1}=\partial \Omega
\cap B_{R}^{c},\quad \Sigma =\Omega \cap \partial B_{R}
\end{equation*}%
and let $\sigma $ denote the $(n-1)$-dimensional measure induced by $\mu $,
that is, $\sigma $ has density $\mathrm{e}^{-\frac{1}{\left\vert
x\right\vert ^{\alpha }}}$ with respect to the $\left( n-1\right) $%
-Hausdorff measure $\mathcal{H}_{n-1}$. First observe that 
\begin{equation}
\sigma (\Gamma _{1})\geq \sigma (\Sigma ).  \label{Equ:iso-ineq-global-10}
\end{equation}%
Indeed, consider the projection $\Pi :x\mapsto \frac{Rx}{\left\vert
x\right\vert }$ of $\Gamma _{1}$ onto $\partial B_{R}.$ Clearly, the image $%
\Pi \left( \Gamma _{1}\right) $ covers $\Sigma $. Since $\Gamma _{1}$ lies
outside $B_{R}$, the mapping $\Pi $ reduces the measure $\mathcal{H}_{n-1}$,
and since the weight function $\mathrm{e}^{-\frac{1}{\left\vert x\right\vert
^{\alpha }}}$ is increasing in $\left\vert x\right\vert $, the same
reduction holds a fortiori for the measure $\sigma $, which proves (\ref%
{Equ:iso-ineq-global-10}).

\begin{figure}[tph]
\begin{center}
\includegraphics[width=0.8\textwidth]{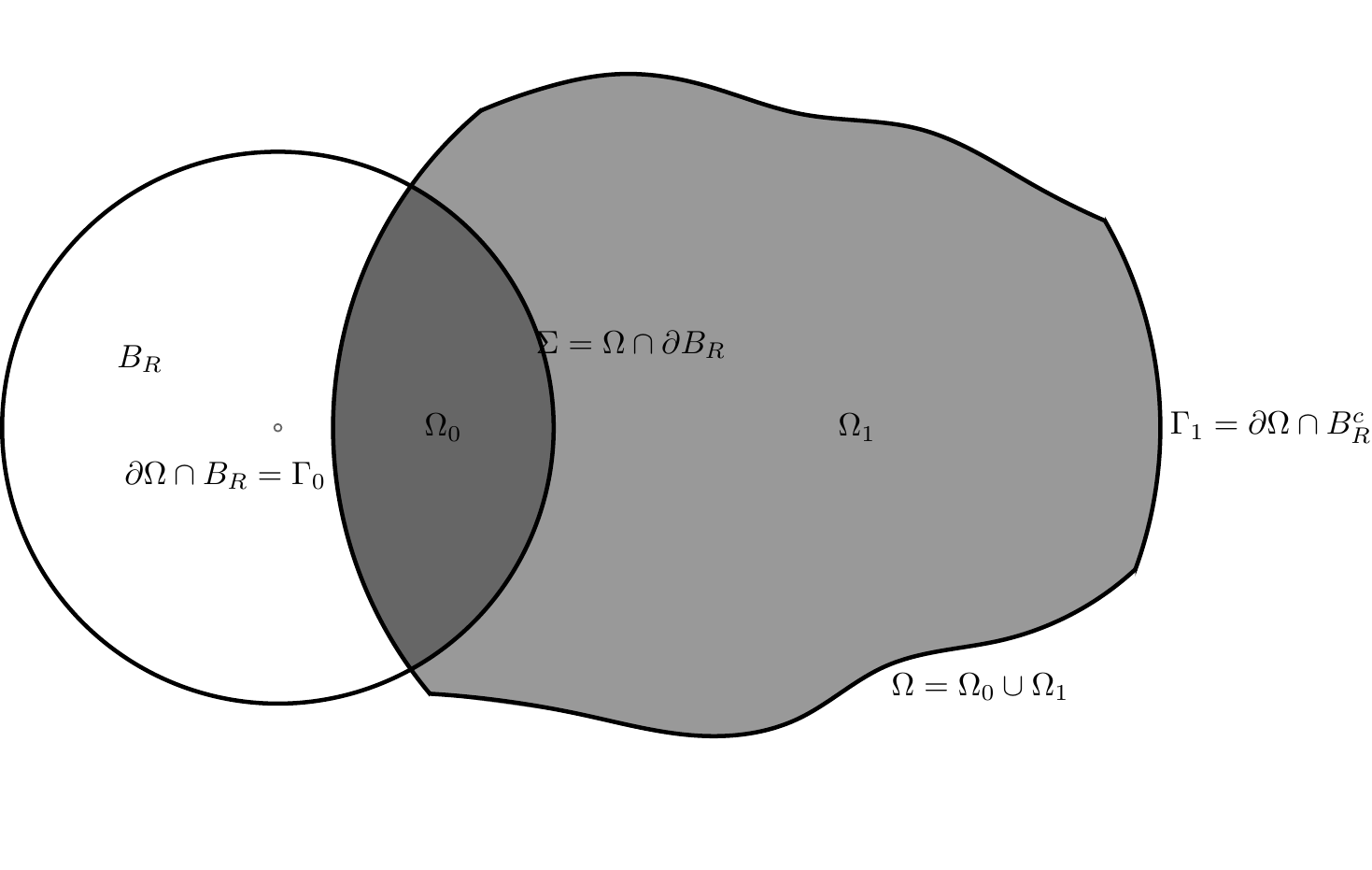}
\end{center}
\caption{Decompostion of $\Omega$ and $\partial \Omega$}
\end{figure}

By (\ref{Equ:iso-ineq-global-1}) we have 
\begin{equation}
\sigma (\Gamma _{0})+\sigma (\Sigma )=\mu ^{+}(\Omega _{0})\geq C_{0}I\left(
\mu (\Omega _{0})\right) .  \label{Equ:iso-ineq-global-12}
\end{equation}%
By (\ref{Equ:iso-ineq-global-2}) we have 
\begin{equation}
\sigma (\Gamma _{1})+\sigma (\Sigma )=\mu ^{+}(\Omega _{1})\geq C_{1}(\mu
(\Omega _{1}))^{\frac{n-1}{n}}.  \label{Equ:iso-ineq-global-11}
\end{equation}%
%
%
Adding up (\ref{Equ:iso-ineq-global-12}) and (\ref{Equ:iso-ineq-global-11})
and replacing $\sigma \left( \Sigma \right) $ by $\sigma \left( \Gamma
_{1}\right) $ according to (\ref{Equ:iso-ineq-global-10}), we obtain 
\begin{equation*}
\sigma \left( \Gamma _{0}\right) +3\sigma (\Gamma _{1})\geq C_{0}I\left( \mu
(\Omega _{0})\right) +C_{1}(\mu (\Omega _{1}))^{\frac{n-1}{n}},
\end{equation*}%
whence (\ref{Equ:iso-ineq-global-5}) follows, as $\mu ^{+}\left( \Omega
\right) =\sigma \left( \Gamma _{0}\right) +\sigma \left( \Gamma _{1}\right) .
$

Now from (\ref{Equ:iso-ineq-global-5}) we deduce the required isoperimetric
inequality, that is, 
\begin{equation}
\mu ^{+}\left( \Omega \right) \geq C\tilde{I}\left( \mu \left( \Omega
\right) \right) .  \label{Imu}
\end{equation}
Set $\tau =\mu (B_{R})$ and consider three cases.

$\left( a\right) $\ Assume that $0\leq \mu (\Omega )\leq \tau $. Clearly,
there is a constant $C_{2}>0$ such that 
\begin{equation}
v^{\frac{n-1}{n}}\geq C_{2}I\left( v\right) \ \text{for all }0\leq v\leq
\tau .  \label{Equ:iso-ineq-global-3}
\end{equation}%
From (\ref{Equ:iso-ineq-global-5}) and (\ref{Equ:iso-ineq-global-3}) we
obtain 
\begin{equation*}
\begin{split}
3\mu ^{+}(\Omega )& \geq C_{0}I\left( \mu (\Omega _{0})\right)
+C_{2}C_{1}I\left( \mu (\Omega _{1})\right) \\
& \geq CI\left( \mu \left( \Omega _{0}\right) \vee \mu \left( \Omega
_{1}\right) \right) \\
& \geq CI\left( \frac{1}{2}\mu \left( \Omega \right) \right) \\
& \geq \frac{1}{2}CI\left( \mu \left( \Omega \right) \right),
\end{split}%
\end{equation*}%
where $C=(C_{0}\wedge \left( C_{1}C_{2}\right) )$ and we have used that $%
I\left( \frac{1}{2}v\right) \geq \frac{1}{2}I\left( v\right) $. Renaming $%
\frac{1}{2}C$ by $C$, we obtain (\ref{Imu}).

$\left( b\right) $\ Assume that $\mu (\Omega )\geq 2\tau $. Since $\mu
\left( \Omega _{0}\right) \leq \tau $, we have in this case 
\begin{equation}
\mu (\Omega _{1})\geq \frac{1}{2}\mu (\Omega ).
\label{Equ:iso-ineq-global-7}
\end{equation}%
Therefore, we obtain from (\ref{Equ:iso-ineq-global-5}) 
\begin{equation}
\mu ^{+}(\Omega )\geq \frac{1}{3}C_{1}\mu (\Omega _{1})^{\frac{n-1}{n}}\geq 
\frac{1}{3}C_{1}\left( \frac{1}{2}\mu (\Omega )\right) ^{\frac{n-1}{n}}=C\mu
(\Omega )^{\frac{n-1}{n}},  \label{Equ:iso-ineq-global-8}
\end{equation}%
with $C=C_{1}\frac{1}{3}\left( \frac{1}{2}\right) ^{\frac{n-1}{n}},$ which
proves (\ref{Imu}) in this case.

$\left( c\right) $\ Assume that $\tau \leq \mu (\Omega )\leq 2\tau $. In
this case we have either $\mu \left( \Omega _{0}\right) \geq \frac{\tau }{2}$
or $\mu \left( \Omega _{1}\right) \geq \frac{\tau }{2}.$ In both cases, 
from (\ref{Equ:iso-ineq-global-5}) we obtain that%
\begin{equation*}
\mu ^{+}\left( \Omega \right) \geq C_{0}I\left( \frac{\tau }{2}\right)
\wedge C_{1}\left( \frac{\tau }{2}\right) ^{\frac{n-1}{n}}=C\left( 2\tau
\right) ^{\frac{n-1}{n}}\geq C\mu \left( \Omega \right) ^{\frac{n-1}{n}},
\end{equation*}%
where the constant $C$ is defined by the middle identity.

Hence, (\ref{Imu}) is satisfied in all cases, which was to be proved.
\end{proof}

\section{An upper bound of the heat kernel}

\label{secUpper} We use the following result from \cite{Gri94_RMI}.

\begin{theorem}
\label{estimates:ED:thm1} \emph{(\cite[Theorem 2.1]{Gri94_RMI}) } Let $%
\left( M,\mu \right) $ be a weighted manifold and assume that $\left( M,\mu
\right) $ satisfies the Faber-Krahn inequality \emph{(\ref%
{Lambda-FK:ineq-iso})} with a function $\Lambda $, where $\Lambda \colon
(0,+\infty )\rightarrow \left( 0,+\infty \right) $ is a decreasing function
such that 
\begin{equation}
\int_{0+}\frac{dv}{v\Lambda (v)}<\infty .
\label{Equ:integral-condition-Lambda}
\end{equation}%
Then the heat kernel $p_{t}\left( x,y\right) $ of $\left( M,\mu \right) $
satisfies the following upper bound%
\begin{equation}
\sup_{x,y\in M}p_{t}(x,y)\leq \frac{4}{V\left( \frac{1}{2}t\right) }
\label{Equ:estimate-E_D}
\end{equation}
for all $t>0$, where the function $V$ is defined by 
\begin{equation}
t=\int_{0}^{V(t)}\frac{dv}{v\Lambda (v)}.  \label{VW_Def:equ-iso}
\end{equation}
\end{theorem}

Combining this theorem with the isoperimetric inequality (\ref%
{Equ:iso-ineq-global}), we obtain the following result.

\begin{theorem}
\label{Thm:upper-bound-hk} Set $M=\mathbb{R}^{n}\setminus \left\{ 0\right\} $
and consider the measure $d\mu =e^{-\frac{1}{\left\vert x\right\vert
^{\alpha }}}dx$ on $M$ for some $\alpha >0$. Then there are positive
constants $C,C_{0}$, depending only on $n$ and $\alpha $, such that the heat
kernel of $\left( M,\mu \right) $ satisfies the following inequality 
\begin{equation}
\sup_{x,y\in M} p_{t}(x,y)\leq C\left\{ 
\begin{array}{ll}
\exp \left( \frac{C_{0}}{t^{\frac{\alpha }{\alpha +2}}}\right) , & 0<t<1, \\ 
t^{-n/2}, & t>1.%
\end{array}%
\right.  \label{Equ:upper-bound-hk}
\end{equation}
\end{theorem}

\begin{proof}
By Cheeger's inequality, the isoperimetric inequality (\ref%
{Equ:iso-ineq-global}) implies the Faber-Krahn inequality with the function 
\begin{equation}
\Lambda (v)=\frac{1}{4}\left( \frac{\tilde{I}(v)}{v}\right) ^{2}=C\left\{ 
\begin{array}{ll}
\displaystyle \left( \log \frac{1}{v}\right) ^{2+\frac{2}{\alpha }},\quad & 
0<v\leq \tau , \\ 
v^{-\frac{2}{n}},\quad & v>\tau%
\end{array}%
\right.  \label{Equ:Lambda-Faber-Krahn1}
\end{equation}%
(cf. \cite[Prop. 2.4]{Gri94_RMI}). Observe that the function in (\ref%
{Equ:Lambda-Faber-Krahn1}) satisfies condition (\ref%
{Equ:integral-condition-Lambda}) so that Theorem \ref{estimates:ED:thm1}
applies and yields the upper bound (\ref{Equ:estimate-E_D}) of $E_{D}\left(
t,x\right) .$ Let us estimate the function $V\left( t\right) $ that enters
the right hand side of (\ref{Equ:estimate-E_D}).

For small enough $t>0$ by (\ref{VW_Def:equ-iso}) we have 
\begin{equation*}
t=\int_{0}^{V(t)}\frac{dv}{v\Lambda (v)}=-\frac{1}{C}\int_{0}^{V(t)}\frac{%
d\log {1/v}}{\left( \log \frac{1}{v}\right) ^{2+\frac{2}{\alpha }}}=\frac{(1+%
\frac{2}{\alpha })}{C}\left( \log \frac{1}{V(t)}\right) ^{-\left( 1+\frac{2}{%
\alpha }\right) },
\end{equation*}%
whence%
\begin{equation*}
V(t)=\exp \left( -\frac{C_{0}}{t^{\frac{\alpha }{\alpha +2}}}\right),
\end{equation*}%
where $C_{0}=C_{0}\left( C,\alpha \right) >0.$ For a large enough $t$ we have%
\begin{equation*}
t=\int_{0}^{V(t)}\frac{dv}{v{\Lambda (v)}}\approx \int_{0}^{V(t)}\frac{dv}{%
v^{1-\frac{2}{n}}}\approx V(t)^{-\frac{2}{n}},
\end{equation*}%
whence 
\begin{equation*}
V(t)\approx t^{\frac{n}{2}}.
\end{equation*}%
Substituting these estimates of $V$ into (\ref{Equ:estimate-E_D}) we obtain (%
\ref{Equ:upper-bound-hk}) for small and large values of $t$. Then the
estimate for the intermediate values of $t$ follows from the fact that the
function $t\mapsto \sup_{x,y\in M}p_{t}(x,y)$ is decreasing.
\end{proof}

\section{A lower bound of the heat kernel}

\label{SecLow}In order to obtain a lower bound of the heat kernel, we use
the following notion. We say that a weighted manifold $\left( M,\mu \right) $
satisfies an anti-Faber-Krahn inequality if, for any $v>0$, there is an open
set $\Omega _{v}\subset M$ such that $\mu (\Omega _{v})=v$ and 
\begin{equation}
\lambda _{1}(\Omega _{v})\leq \Lambda (v).  \label{Equ:anti-Faber-Krahn}
\end{equation}
We shall use the following result from \cite{CG97}.

\begin{theorem}
\label{Thm:sup-lower-heat-kernel}\emph{(\cite[Theorem 3.2]{CG97})} Let $%
\Lambda $ be a function as in Theorem \emph{\ref{estimates:ED:thm1}}. Assume
that $\left( M,\mu \right) $ satisfies an anti-Faber-Krahn inequality with
the function $\Lambda $. Define a function $\gamma :\mathbb{R}%
_{+}\rightarrow \mathbb{R}_{+}$ by the identity 
\begin{equation}
t=\int_{0}^{\gamma (t)}\frac{dv}{v\Lambda (v)}  \label{Equ:LambdaV-gamma_t}
\end{equation}%
and assume that $\gamma (t)$ satisfies the following property: there exists
some constant $c_{\gamma }>0$ such that 
\begin{equation}
\frac{\gamma ^{\prime }(s)}{\gamma (s)}\geq C_{\gamma }\frac{\gamma ^{\prime
}(t)}{\gamma (t)},\quad \text{for all }0<t\leq s\leq 2t.
\label{Equ:D-property}
\end{equation}%
Then, for all $t>0$, 
\begin{equation}
\sup_{x\in M}p_{t}(x,x)\geq \frac{1}{\gamma \left( \frac{2}{c_{\gamma }}%
t\right) }.  \label{Equ:sup-lower-heat-kernel}
\end{equation}
\end{theorem}

To apply Theorem \ref{Thm:sup-lower-heat-kernel} we need the following lemma.

\begin{lemma}
\label{lmm:lambda_1_Br-estimate}Consider the manifold $M=\mathbb{R}%
^{n}\setminus \left\{ 0\right\} $ with measure $d\mu =\mathrm{e}^{-\frac{1}{%
\left\vert x\right\vert ^{\alpha }}}dx$ where $\alpha >0$. For any $r>0$ set 
\begin{equation*}
B_{r}:=\{x\in \mathbb{R}^{n}\setminus \{0\}\colon \left\vert x\right\vert
<r\}.
\end{equation*}%
There exists some constant $C>0$ such that for all $0<r<1$, 
\begin{equation}
\lambda _{1}(B_{r})\leq Cr^{-2(1+\alpha )}
\label{Equ:lambda_1_Br-estimate-small-r}
\end{equation}%
(in fact $\lambda _{1}(B_{r})\approx r^{-2(1+\alpha )}$) and for all $r\geq 1
$ 
\begin{equation}
\lambda _{1}(B_{r})\leq Cr^{-2}  \label{Equ:lambda_1_Br-estimate-large-r}
\end{equation}%
(in fact $\lambda _{1}\left( B_{r}\right) \approx r^{-2}$).
\end{lemma}

\begin{proof}
Let us first prove (\ref{Equ:lambda_1_Br-estimate-large-r}). Fix $r\geq 1$
and consider a test function%
\begin{equation*}
\varphi \left( x\right) =\left\{ 
\begin{array}{ll}
\left( \left\vert x\right\vert -r/4\right) _{+}, & \left\vert x\right\vert
\leq r/2, \\ 
\left( r-\left\vert x\right\vert \right) _{+} & \left\vert x\right\vert \geq
3r/4, \\ 
\frac{1}{4}r, & r/2<\left\vert x\right\vert <3r/4,%
\end{array}%
\right.
\end{equation*}%
that is a Lipschitz function with compact support in $B_{r}$. By the
variational principle, we have%
\begin{equation*}
\lambda _{1}\left( B_{R}\right) \leq \frac{\int_{M}\left\vert \nabla \varphi
\right\vert ^{2}d\mu }{\int_{M}\varphi ^{2}d\mu }.
\end{equation*}%
Clearly, we have%
\begin{equation*}
\int_{M}\varphi ^{2}d\mu \geq \int_{B_{r/2}\setminus B_{r/4}}\varphi
^{2}d\mu =\left( \frac{1}{4}r\right) ^{2}\mu \left( B_{r/2}\setminus
B_{r/4}\right) \approx r^{2}r^{n}=r^{n+2},
\end{equation*}%
where we use the fact that outside $B_{r/4}$ the measure $\mu $ is finitely
proportional to the Lebesgue measure. Also, since $\left\vert \nabla \varphi
\right\vert \leq 1$, we have%
\begin{equation*}
\int_{M}\left\vert \nabla \varphi \right\vert ^{2}d\mu \leq \mu \left(
B_{r}\right) \leq Cr^{n}.
\end{equation*}%
Combining this with the previous line, we obtain (\ref%
{Equ:lambda_1_Br-estimate-large-r}).

Let us now prove (\ref{Equ:lambda_1_Br-estimate-small-r}). Set $S(r)=\mu
^{+}(B_{r})$ and $V(r)=\mu (B_{r})$. By \cite[Theorem 2.10]{Gri06_JW06} (see
also \cite{Gri99_OT}) we have 
\begin{equation}
\lambda _{1}(B_{r})\approx \frac{1}{F(r)}  \label{Equ:lambda_1-1/Fr}
\end{equation}%
for all $r>0$, where 
\begin{equation*}
F(r):=\sup_{0<\xi <r}\left[ V(\xi )\int_{\xi }^{r}\frac{dt}{S(t)}\right] .
\end{equation*}%
By definition of $\mu $ we have 
\begin{equation*}
S(r)=\omega _{n}r^{n-1}\mathrm{e}^{-\frac{1}{r^{\alpha }}}
\end{equation*}%
and 
\begin{equation*}
V(r)=\int_{0}^{r}S(t)\,dt\approx r^{n+\alpha }\mathrm{e}^{-\frac{1}{%
r^{\alpha }}}.
\end{equation*}%
Let us show that there exists some $c>0$ such that for $0<r<1$ 
\begin{equation*}
V\left( \frac{r}{2}\right) \int_{r/2}^{r}\frac{dt}{S\left( t\right) }\geq
cr^{2\left( 1+\alpha \right) },
\end{equation*}%
which would imply $F\left( r\right) \geq cr^{2\left( 1+\alpha \right) }$
and, hence, (\ref{Equ:lambda_1_Br-estimate-small-r}).

Set $\xi =r/2$ and observe that 
\begin{equation*}
\int_{\xi }^{r}\frac{1}{S(t)}\,dt=\frac{1}{\omega _{n}}\int_{\xi
}^{r}t^{1-n}\exp \left( \frac{1}{t^{\alpha }}\right) \,dt\approx
r^{1-n}\int_{\xi }^{r}\exp \left( \frac{1}{t^{\alpha }}\right) \,dt,
\end{equation*}%
whence%
\begin{equation}
V(\xi )\int_{\xi }^{r}\frac{1}{S(t)}\,dt\approx r^{1+\alpha }\exp \left( -%
\frac{1}{\xi ^{\alpha }}\right) \int_{\xi }^{r}\exp \left( \frac{1}{%
t^{\alpha }}\right) \,dt.  \label{Equ:V-int-estimate1}
\end{equation}
Next let us verify that%
\begin{equation}
\exp \left( \frac{1}{\left( \xi +\xi ^{1+\alpha }\right) ^{\alpha }}\right)
\geq C^{-1}\exp \left( \frac{1}{\xi ^{\alpha }}\right)
\label{Equ:exp-t-alpha}
\end{equation}%
for some $C>0$. Indeed, 
\begin{equation*}
\frac{\exp \left( \frac{1}{\xi ^{\alpha }}\right) }{\exp \left( \frac{1}{%
\left( \xi +\xi ^{1+\alpha }\right) ^{\alpha }}\right) }=\exp \left( \frac{1%
}{\xi ^{\alpha }}-\frac{1}{\xi ^{\alpha }\left( 1+\xi ^{\alpha }\right)
^{\alpha }}\right) =\exp \left( \frac{\left( 1+\xi ^{\alpha }\right)
^{\alpha }-1}{\xi ^{\alpha }\left( 1+\xi ^{\alpha }\right) ^{\alpha }}%
\right) .
\end{equation*}%
Since the function $x\mapsto \frac{\left( 1+x\right) ^{\alpha }-1}{x}$ is
bounded for $x\in \left( 0,1\right) $, say by a constant $C$, we obtain%
\begin{equation*}
\frac{\exp \left( \frac{1}{\xi ^{\alpha }}\right) }{\exp \left( \frac{1}{%
\left( \xi +\xi ^{1+\alpha }\right) ^{\alpha }}\right) }\leq \exp \left( 
\frac{C}{\left( 1+\xi ^{\alpha }\right) ^{\alpha }}\right) \leq \exp \left(
C\right) ,
\end{equation*}%
which proves (\ref{Equ:exp-t-alpha}). Since $r\geq \xi +\xi ^{1+\alpha }$,
it follows that 
\begin{equation*}
\int_{\xi }^{r}\exp \left( \frac{1}{t^{\alpha }}\right) \,dt\geq \int_{\xi
}^{\xi +\xi ^{1+\alpha }}\exp \left( \frac{1}{t^{\alpha }}\right) \,dt\geq
\xi ^{1+\alpha }\exp \left( \frac{1}{\left( \xi +\xi ^{1+\alpha }\right)
^{\alpha }}\right) \geq C^{-1}\xi ^{1+\alpha }\exp \left( \frac{1}{\xi
^{\alpha }}\right) .
\end{equation*}%
Substituting the estimate above into (\ref{Equ:V-int-estimate1}) we obtain
that, for some constant $C_{1}>0$, 
\begin{equation*}
V(\xi )\int_{\xi }^{r}\frac{1}{S(t)}\,dt\geq C_{1}r^{1+\alpha }\exp \left( -%
\frac{1}{\xi ^{\alpha }}\right) C^{-1}\xi ^{1+\alpha }\exp \left( \frac{1}{%
\xi ^{\alpha }}\right) \approx r^{2(1+\alpha )},
\end{equation*}%
which finishes the proof of (\ref{Equ:lambda_1_Br-estimate-small-r}).
\end{proof}

\bigskip

Finally we can prove a lower bound of the heat kernel.

\begin{theorem}
\label{Thm:sup-lower-heat-kernel-apply}For the manifold $M=\mathbb{R}%
^{n}\setminus \left\{ 0\right\} $ with measure $d\mu =\mathrm{e}^{-\frac{1}{%
\left\vert x\right\vert ^{\alpha }}}dx$, there exist constants $c,c_{0}>0$
depending on $n$ and $\alpha $, such that the heat kernel of $\left( M,\mu
\right) $ satisfies the following estimate 
\begin{equation}
\sup_{x\in M}p_{t}(x,x)\geq c\left\{ 
\begin{array}{ll}
\exp \left( \frac{c_{0}}{t^{\frac{\alpha }{2+\alpha }}}\right) , & 0<t<1, \\ 
t^{-n/2}, & t\geq 1.%
\end{array}%
\right.  \label{Equ:sup-lower-heat-kernel-apply}
\end{equation}
\end{theorem}

\begin{proof}
For any $v>0$, take $\Omega _{v}=B_{r}$ where $r$ is chosen so that $\mu
(B_{r})=v$. If $v$ is small enough then by (\ref{Equ:v-approx-R}) we have 
\begin{equation*}
r^{-1}\approx \left( \log \frac{1}{v}\right) ^{\frac{1}{\alpha }}.
\end{equation*}%
Hence by Lemma \ref{lmm:lambda_1_Br-estimate} we obtain 
\begin{equation*}
\lambda _{1}(\Omega _{v})\leq Cr^{-2\left( 1+\alpha \right) }\leq C^{\prime
}\left( \log \frac{1}{v}\right) ^{\frac{2(1+\alpha )}{\alpha }}:=\Lambda (v).
\end{equation*}%
As in the proof of Theorem \ref{Thm:upper-bound-hk}, the function $\gamma $
from (\ref{Equ:LambdaV-gamma_t}) has the expression 
\begin{equation*}
\gamma (t)=\exp \left( -\frac{C_{0}}{t^{\frac{\alpha }{2+\alpha }}}\right)
\end{equation*}%
for some $C_{0}>0.$ It is easy to verify that this function $\gamma $
satisfies property (\ref{Equ:D-property}). Hence by Theorem \ref%
{Thm:sup-lower-heat-kernel} we obtain the lower bound (\ref%
{Equ:sup-lower-heat-kernel}) for small values of $t$. The case of large
values of $t$ is treated similarly.
\end{proof}

\begin{acknowledgment}
We would like to thank Yuri Kondratiev for stimulating discussions that
motivated this work.
\end{acknowledgment}

\def\cprime{$'$}
\providecommand{\bysame}{\leavevmode\hbox to3em{\hrulefill}\thinspace}
\providecommand{\MR}{\relax\ifhmode\unskip\space\fi MR }
\providecommand{\MRhref}[2]{%
  \href{http://www.ams.org/mathscinet-getitem?mr=#1}{#2}
}
\providecommand{\href}[2]{#2}

\end{document}